\documentclass[11pt]{article}
\usepackage{a4wide}

\usepackage{amsmath,amsfonts,amssymb,amsthm,amscd}
\usepackage{epsfig}

\usepackage{color}

\usepackage{hyperref}
\usepackage{breakurl}

\definecolor{modra3}{rgb}{.1,.0,.4}
\hypersetup{colorlinks=true, linkcolor=modra3, urlcolor=modra3, citecolor=modra3, pdfpagemode=UseNone, pdfstartview=}

\newtheorem{theorem}{Theorem}        
\newtheorem{proposition}[theorem]{Proposition}            
\newtheorem{lemma}[theorem]{Lemma}

\newtheorem{observation}[theorem]{Observation}

\begin{document}

\title{Simple realizability of complete abstract topological graphs simplified\thanks{The author was supported 
by Swiss National Science Foundation Grants 200021-137574 and 200020-144531, by the ESF Eurogiga project GraDR as GA\v{C}R GIG/11/E023, by the grant no.
18-19158S of the Czech Science Foundation (GA\v{C}R), by the grant GAUK 1262213 of the Grant Agency of Charles University, by Charles University project UNCE/SCI/004, and by the PRIMUS/17/SCI/3 project of Charles University. An extended abstract of this paper appeared in the proceedings of Graph Drawing 2015.
} 
} 

\author{
Jan Kyn\v cl
} 

\date{}

\maketitle

\begin{center}
{\footnotesize
Department of Applied Mathematics and Institute for Theoretical Computer Science, \\
Charles University, Faculty of Mathematics and Physics, \\
Malostransk\'e n\'am.~25, 118~00~ Praha 1, Czech Republic; \\
\texttt{kyncl@kam.mff.cuni.cz}
\\\ \\
\'Ecole Polytechnique F\'ed\'erale de Lausanne, Chair of Combinatorial Geometry, \\
EPFL-SB-MATHGEOM-DCG, Station 8, CH-1015 Lausanne, Switzerland\\
}
\end{center}  

\begin{abstract}
An {\em abstract topological graph\/} (briefly an {\em AT-graph}) is a pair $A=(G,\mathcal{X})$ where $G=(V,E)$ is a graph and $\mathcal{X}\subseteq {E \choose 2}$ is a set of pairs of its edges. The AT-graph $A$ is {\em simply realizable\/} if $G$ can be drawn in the plane so that each pair of edges from $\mathcal{X}$ crosses exactly once and no other pair crosses. We show that simply realizable complete AT-graphs are characterized by a finite set of forbidden AT-subgraphs, each with at most six vertices. This implies a straightforward polynomial algorithm for testing simple realizability of complete AT-graphs, which simplifies a previous algorithm by the author. 
We also show an analogous result for independent $\mathbb{Z}_2$-realizability, where only the parity of the number of crossings for each pair of independent edges is specified.
\end{abstract}


\section{Introduction}
A {\em topological graph\/} $T=(V(T), E(T))$ is a drawing of a graph $G$ in the plane such that the vertices of $G$ are represented by a set $V(T)$ of distinct points and the edges of $G$ are represented by a set $E(T)$ of simple curves connecting the corresponding pairs of points. We call the elements of $V(T)$ and $E(T)$ the {\em vertices\/} and the {\em edges\/} of $T$, respectively. The drawing has to satisfy the following general position conditions: (1) the edges pass through no vertices except their endpoints, (2) every pair of edges has only a finite number of intersection points, (3) every intersection point of two edges is either a common endpoint or a proper crossing (``touching'' of the edges is not allowed), and (4) no three edges pass through the same crossing. A topological graph or a drawing is {\em simple\/} if every pair of edges has at most one common point, which is either a common endpoint or a crossing. Simple topological graphs appear naturally as crossing-minimal drawings: it is well 
known that if two edges in a topological graph have more than one common point, then there is a local redrawing that decreases the total number of crossings.
A topological graph is {\em complete\/} if it is a drawing of a complete graph. 

An {\em abstract topological graph\/} (briefly an {\em AT-graph}), a notion introduced by Kratochv\'\i l, Lubiw and Ne\v{s}et\v{r}il~\cite{K91_noncrossing}, 
is a pair $(G,\mathcal{X})$ where $G$ is a graph and $\mathcal{X} \subseteq {E(G) \choose 2}$ is a set of pairs of its edges. Here we assume that $\mathcal{X}$ consists only of independent (that is, nonadjacent) pairs of edges. For a simple topological graph $T$ that is a drawing of $G$, let $\mathcal{X}_T$ be the set of pairs of edges having a common crossing.
A simple topological graph $T$ is a {\em simple realization\/} of $(G,\mathcal{X})$ if $\mathcal{X}_T=\mathcal{X}$. We say that $(G,\mathcal{X})$ is {\em simply realizable} if $(G,\mathcal{X})$ has a simple realization.

An AT-graph $(H,\mathcal{Y})$ is an {\em AT-subgraph} of an AT-graph $(G,\mathcal{X})$ if $H$ is a subgraph of $G$ and $\mathcal{Y}=\mathcal{X}\cap{E(H) \choose 2}$. Clearly, a simple realization of $(G,\mathcal{X})$ restricted to the vertices and edges of $H$ is a simple realization of $(H,\mathcal{Y})$.

We are ready to state our main result.

\begin{theorem}\label{veta_konecna_charakterizace}
Every complete AT-graph that is not simply realizable has an AT-subgraph on at most six vertices that is not simply realizable.
\end{theorem}

We also show that AT-subgraphs with five vertices are not sufficient to characterize simple realizability.

\begin{theorem}\label{veta_protipriklad}
There is a complete AT-graph $A$ with six vertices such that all its complete AT-subgraphs with five vertices are simply realizable, but $A$ itself is not.
\end{theorem}

Theorems~\ref{veta_konecna_charakterizace} and~\ref{veta_protipriklad} are proved in Sections~\ref{section_konecna_charakterizace} and~\ref{section_protipriklad}, respectively.

Theorem~\ref{veta_konecna_charakterizace} implies a straightforward polynomial algorithm for simple realizability of complete AT-graphs, running in time $O(n^{6})$ for graphs with $n$ vertices. It is likely that this running time can be improved relatively easily. However, compared to the first polynomial algorithm for simple realizability of complete AT-graphs~\cite{K11_simple_real}, the new algorithm may be more suitable for implementation and for practical applications, such as generating all simply realizable complete AT-graphs of given size or computing the crossing number of the complete graph~\cite{Ch11_facets, M08_recent}. On the other hand,  the new algorithm does not directly provide the drawing itself, unlike the original algorithm~\cite{K11_simple_real}. The explicit list of realizable AT-graphs on six vertices can be generated using the database of small simple complete topological graphs created by \'Abrego et al.~\cite{Aetal15_allgood}.

For general noncomplete graphs, no such finite characterization by forbidden AT-subgraphs is possible. Indeed, in the special case when $\mathcal{X}$ is empty, the problem of simple realizability is equivalent to planarity, and there are minimal nonplanar graphs of arbitrarily large girth; for example, subdivisions of $K_5$. Moreover, simple realizability for general AT-graphs is NP-complete, even when the underlying graph is a matching~\cite{K91_stringII,KM89_np}. See~\cite{K11_simple_real} for an overview of other similar realizability problems.

The proof of Theorem~\ref{veta_konecna_charakterizace} is based on the polynomial algorithm for simple realizability of complete AT-graphs from~\cite{K11_simple_real}. The main idea is very simple: every time the algorithm rejects the input, it is due to an obstruction of constant size.

Theorem~\ref{veta_konecna_charakterizace} is an analogue of a similar characterization of simple monotone drawings of $K_n$ by forbidden $5$-tuples, and pseudolinear drawings of $K_n$ by forbidden $4$-tuples~\cite{BFK15_monotone}.

It is known that the set of pairs of crossing edges in a simple complete topological graph determines its rotation system, up to inversion~\cite{G05_complete,K11_simple_real}, and vice versa~\cite{K09_enumeration,PT04_how}; see Proposition~\ref{prop_poradi_na_hvezde} for an extension of this correspondence and Section~\ref{sec_preliminaries} for the definitions.
\'Abrego et al.~\cite{Aetal15_allgood,A14_pers} independently verified that simple complete topological graphs with up to nine vertices can be characterized by forbidden rotation systems of five-vertex subgraphs. They conjectured that the same characterization is true for all simple complete topological graphs~\cite{A14_pers}. 
This conjecture now follows by combining their result for six-vertex graphs with Theorem~\ref{veta_konecna_charakterizace}.
This gives a finite characterization of \emph{realizable abstract rotation systems} defined in~\cite[Sect. 3.5]{K13_improved}, where it was also stated that such a characterization was not likely~\cite[p. 739]{K13_improved}. The fact that only $5$-tuples are sufficient for the characterization by rotation systems should perhaps not be too surprising, as rotation systems characterize simple drawings of $K_n$ more economically, using only $O(n^2\log n)$ bits, whereas AT-graphs need $\Theta(n^4)$ bits.

\subsubsection*{Independent $\mathbb{Z}_2$-realizability}
A topological graph $T$ is an {\em independent $\mathbb{Z}_2$-realization\/} of an AT-graph $(G,\mathcal{X})$ if $\mathcal{X}$ is the set of pairs of independent edges that cross an odd number of times in $T$. We say that $(G,\mathcal{X})$ is {\em independently $\mathbb{Z}_2$-realizable} if $(G,\mathcal{X})$ has an independent $\mathbb{Z}_2$-realization. 

Clearly, every simple realization of an AT-graph is also its independent $\mathbb{Z}_2$-realization. The converse is not true, since every simple realization of $K_4$ has at most one crossing, but there are independently $\mathbb{Z}_2$-realizable AT-graphs $(K_4,\mathcal{X})$ with $|\mathcal{X}|=2$ or $|\mathcal{X}|=3$. 
%
Thus, independent $\mathbb{Z}_2$-realizability is only a necessary condition for simple realizability. However, independent $\mathbb{Z}_2$-realizability of arbitrary AT-graphs can be tested in polynomial time since it is equivalent to the solvability of a system of linear equations over $\mathbb{Z}_2$; see Section~\ref{section_Z2} for details.

Independent $\mathbb{Z}_2$-realizability has been usually considered only in the special case when $\mathcal{X}=\emptyset$. In this case, for every graph $G$, the AT-graph $(G,\emptyset)$ has an independent $\mathbb{Z}_2$-realization if and only if $(G,\emptyset)$ has a simple realization, and this is equivalent to $G$ being planar. This fact is known as the (strong) Hanani--Tutte theorem~\cite{Ha34_uber,Tutte70_toward}. A related concept, the \emph{independent odd crossing number} of a graph $G$, $\text{iocr}(G)$, measuring the minimum cardinality of $\mathcal{X}$ for which $(G,\mathcal{X})$ has an independent $\mathbb{Z}_2$-realization, has been introduced by Sz\'ekely~\cite{Sze04_iocr}. The asymptotic value of $\text{iocr}(K_n)$ is not known, and computing $\text{iocr}(G)$ for a general graph $G$ is NP-complete~\cite{{PSS11_rotation}}. See Schaefer's survey~\cite{Sch14_survey} for more information.

We call an AT-graph $(G,\mathcal{X})$ \emph{even} (or an \emph{even $G$}) if $|\mathcal{X}|$ is even, and \emph{odd} (or an \emph{odd $G$}) if $|\mathcal{X}|$ is odd.
The following theorem is an analogue of Theorem~\ref{veta_konecna_charakterizace} for independent $\mathbb{Z}_2$-realizability.

\begin{theorem}\label{veta_Z2_charakterizace}
Every complete AT-graph that is not independently $\mathbb{Z}_2$-realizable has an AT-subgraph on at most six vertices that is not independently $\mathbb{Z}_2$-realizable. More precisely, a complete AT-graph is independently $\mathbb{Z}_2$-realizable if and only if it contains no even $K_5$ and no odd $2K_3$ as an AT-subgraph.
\end{theorem}

Theorem~\ref{veta_Z2_charakterizace} is proved in Section~\ref{section_Z2}, where we also show that AT-subgraphs with five vertices are not sufficient to characterize independently $\mathbb{Z}_2$-realizable complete AT-graphs. 

Theorem~\ref{veta_Z2_charakterizace} again implies a straightforward algorithm for independent $\mathbb{Z}_2$-realizability of complete AT-graphs, running in time $O(n^{6})$ for graphs with $n$ vertices. This is slightly better compared to the algebraic algorithm, which essentially consists of solving a system of $\Theta(n^4)$ equations with $\Theta(n^3)$ variables; see Section~\ref{section_Z2}. The algorithms by Ibarra, Moran and Hui~\cite{IMH82_matrix} and Jeannerod~\cite{Jea06_LSP_decomp} solve such a system in time $O(n^{3\omega+1})$, where $O(n^{\omega})$ is the complexity of multiplication of two square $n\times n$ matrices. The best current algorithms for matrix multiplication give $\omega<2.3729$~\cite{LeGall14_powers,Williams12_faster}.

\section{Preliminaries}
\label{sec_preliminaries}

For a topological graph $T$ and a subset $U\subseteq V(T)$, by $T[U]$ we denote the topological subgraph of $T$ induced by $U$. Analogously, 
for an AT-graph $A=(G,\mathcal{X})$ and a subset $U\subseteq V(G)$, by $A[U]$ we denote the induced AT-subgraph $(G[U], \mathcal{X}\cap {E(G[U])\choose 2})$.

A {\em face\/} of a topological graph $T$ is a connected component of the set $\mathbb R^2 \setminus \bigcup E(T)$. 

Simple topological graphs $G$ and $H$ are {\em weakly isomorphic\/} if they are simple realizations of the same abstract topological graph. Topological graphs $G$ and $H$ are {\em isomorphic\/} if and only if there exists a homeomorphism of the sphere that transforms $G$ into $H$.

The {\em rotation\/} of a vertex $v$ in a topological graph is the clockwise cyclic order in which the edges incident with $v$ leave the vertex $v$. The {\em rotation system\/} of a topological graph is the set of rotations of all its vertices. Similarly we define the {\em rotation\/} of a crossing $x$ of edges $uv$ and $yz$ as the clockwise order in which the four parts $xu$, $xv$, $xy$ and $xz$ of the edges $uv$ and $yz$ leave the point $x$. Note that each crossing has exactly two possible rotations.
We will represent the rotation of a vertex $v$ as an ordered sequence of the other endpoints of the edges incident with $v$. Similarly, we will represent the rotation of a crossing $x$ as an ordered sequence of the four endpoints of the edges incident with $x$.

The {\em extended rotation system\/} of a topological graph is the set of rotations of all its vertices and crossings. 

Assuming that $T$ and $T'$ are drawings of the same abstract graph, we say that their rotation systems are {\em inverse\/} if for each vertex $v \in V(T)$, the rotation of $v$ and the rotation of the corresponding vertex $v' \in V(T')$ are inverse cyclic permutations. If $T$ and $T'$ are weakly isomorphic simple topological graphs, we say that their extended rotation systems are {\em inverse\/} if their rotation systems are inverse and, in addition, for every crossing $x$ in $T$, the rotation of $x$ and the rotation of the corresponding crossing $x'$ in $T'$ are inverse cyclic permutations. For example, if $T'$ is a mirror image of $T$, then $T$ and $T'$ have inverse extended rotation systems.

Simple complete topological graphs have the following key property.

\begin{proposition}\label{prop_poradi_na_hvezde}{\rm\cite{G05_complete,K11_simple_real}}
\begin{itemize}
\item[$(1)$]
If two simple complete topological graphs are weakly isomorphic, then their extended rotation systems are either the same or inverse.
\item[$(2)$] For every edge $e$ of a simple complete topological graph $T$ and for every pair of edges $f,f'\in E(T)$ that have a common endpoint and cross $e$, the AT-graph of $T$ determines the order of crossings of $e$ with the edges $f,f'$.
\end{itemize}
\end{proposition}

By inspecting simple drawings of $K_4$, it can be shown that the converse of Proposition~\ref{prop_poradi_na_hvezde} also holds: the rotation system of a simple complete topological graph determines which pairs of edges cross~\cite{K09_enumeration,PT04_how}.

We say that two cyclic permutations of sets $A,B$ are {\em compatible} if they are restrictions of a common cyclic permutation of $A\cup B$.

\section{Proof of Theorem~\ref{veta_protipriklad}}
\label{section_protipriklad}
We use the shortcut $ij$ to denote the edge $\{i,j\}$. Let  $A=((V,E),\mathcal{X})$ be the complete AT-graph with vertex set $V=\{0,1,2,3,4,5\}$ and with 
\begin{align*}
\mathcal{X}=\{&\{02,13\},\{02,14\},\{02,15\},\{02,35\},\{03,14\},\{03,15\},\{03,24\},\{04,15\},\\
 &\{04,25\},\{04,35\},\{13,24\},\{24,35\},\{35,14\},\{14,25\},\{25,13\}\}.
\end{align*}

Every complete AT-subgraph of $A$ with five vertices is simply realizable; see Figure~\ref{obr_2_realizations}.

Now we show that $A$ is not simply realizable. Suppose that $T$ is a simple realization of $A$. Without loss of generality, assume that the rotation of $5$ in $T[\{1,2,3,5\}]$ is $(1,2,3)$.  By Proposition~\ref{prop_poradi_na_hvezde} and by the first drawing in Figure~\ref{obr_2_realizations}, the rotation of $5$ in $T[\{1,2,3,4,5\}]$ is $(1,2,3,4)$, since the inverse would not be compatible with $(1,2,3)$. Similarly, by the second drawing in Figure~\ref{obr_2_realizations} the rotation of $5$ in $T[\{0,2,3,4,5\}]$ is $(2,3,0,4)$, since the inverse would not be compatible with $(1,2,3,4)$. By the third drawing in Figure~\ref{obr_2_realizations}, the rotation of $5$ in $T[\{0,1,3,4,5\}]$ is $(0,1,3,4)$ or $(0,4,3,1)$, but neither of them is compatible with both $(1,2,3,4)$ and $(2,3,0,4)$; a contradiction.

\begin{figure}
\begin{center}
\epsfbox{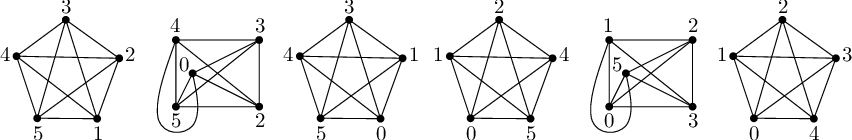}
\end{center}
\caption{Simple realizations of all six complete subgraphs of $A$ with five vertices.}
\label{obr_2_realizations}
\end{figure}


\section{Proof of Theorem~\ref{veta_konecna_charakterizace}}
\label{section_konecna_charakterizace}
Let $A=(K_n,\mathcal{X})$ be a given complete abstract topological graph with vertex set $[n]=\{1,2,\dots,n\}$.
The algorithm from~\cite{K11_simple_real} for deciding simple realizability of $A$ has the following three main steps: computing the rotation system, determining the homotopy class of every edge with respect to the edges incident with one chosen vertex $v$, and computing the number of crossings of every pair of edges in a crossing-optimal drawing
with the rotation system and homotopy class fixed from the previous steps.
We follow the algorithm and analyze each step in detail.


\subsection*{Step 1: computing the extended rotation system}
This step is based on the proof of Proposition~\ref{prop_poradi_na_hvezde}; see~\cite[Proposition 3]{K11_simple_real}.


\subsubsection*{1a) Realizability of $5$-tuples}
For every $5$-tuple $Q$ of vertices of $A$, the algorithm tests whether $A[Q]$ is simply realizable. If not, then the $5$-tuple certifies that $A$ is not simply realizable. If $A[Q]$ is simply realizable, then by Proposition~\ref{prop_poradi_na_hvezde}, the algorithm computes a rotation system $\mathcal{R}(Q)$ such that the rotation system of every simple realization of $A[Q]$ is either $\mathcal{R}(Q)$ or the inverse of $\mathcal{R}(Q)$.


\subsubsection*{1b) Orienting $5$-tuples} 

An \emph{orientation map} $\Phi$ is a map assigning to every rotation system $\mathcal{R}(Q)$ with $Q\in {[n] \choose 5}$ either $\mathcal{R}(Q)$ itself or its inverse $(\mathcal{R}(Q))^{-1}$.

The algorithm selects an orientation map $\Phi$ such that for every pair of $5$-tuples $Q,Q'\in {[n] \choose 5}$ having four common vertices and for each $x\in Q\cap Q'$, the rotations of $x$ in $\Phi(\mathcal{R}(Q))$ and $\Phi(\mathcal{R}(Q'))$ are compatible. If there is no such $\Phi$, the AT-graph $A$ is not simply realizable. We show that in this case there is a set $S$ of six vertices of $A$ that certifies this.

Let $Q_1,Q_2$ be two $5$-tuples with four common elements, let $\mathcal{R}_1$ be a rotation system on $Q_1$ and let $\mathcal{R}_2$ be a rotation system on $Q_2$. We say that $\mathcal{R}_1$ and $\mathcal{R}_2$ are \emph{compatible} if for every $x\in Q_1\cap Q_2$, the rotations of $x$ in $\mathcal{R}_1$ and $\mathcal{R}_2$ are compatible. 

Let $\mathcal{G}$ be the graph with vertex set ${[n]\choose 5}$ and edge set consisting of those pairs $\{Q,Q'\}$ whose intersection has size $4$. For every edge $\{Q,Q'\}$ of $\mathcal{G}$, at most one orientation of $\mathcal{R}(Q')$ is compatible with $\mathcal{R}(Q)$. If no orientation of $\mathcal{R}(Q')$ is compatible with $\mathcal{R}(Q)$, then the $6$-tuple $S=Q\cup Q'$ certifies that $A$ is not simply realizable. We may thus assume that for every edge $\{Q,Q'\}$ of $\mathcal{G}$, exactly one orientation of $\mathcal{R}(Q')$ is compatible with $\mathcal{R}(Q)$. Let $\mathcal{E}$ be the set of those edges $\{Q,Q'\}$ of $\mathcal{G}$ such that $\mathcal{R}(Q)$ and $\mathcal{R}(Q')$ are not compatible.

Call a set $\mathcal{W} \subseteq {[n]\choose 5}$ \emph{orientable} if there is an orientation map $\Phi$ such that for every pair of $5$-tuples $Q,Q' \in \mathcal{W}$ with $|Q\cap Q'|=4$, the rotation systems $\Phi(\mathcal{R}(Q))$ and $\Phi(\mathcal{R}(Q'))$ are compatible. 

\begin{lemma}
\label{lemma_orientable}
If ${[n]\choose 5}$ is not orientable, then there is a $6$-tuple $S \subseteq [n]$ such that ${S\choose 5}$ is not orientable.
\end{lemma}

\begin{proof}

Clearly, ${[n]\choose 5}$ is not orientable if and only if $\mathcal{G}$ has a cycle with an odd number of edges from $\mathcal{E}$. Call such a cycle a \emph{nonorientable} cycle.
We claim that if $\mathcal{G}$ has a nonorientable cycle, then $\mathcal{G}$ has a nonorientable triangle. 
Let $\mathcal{C}(\mathcal{G})$ be the cycle space of $\mathcal{G}$.
The parity of the number of edges of $\mathcal{E}$ in $\mathcal{K}\in\mathcal{C}(\mathcal{G})$ is a linear form on $\mathcal{C}(\mathcal{G})$. Hence, to prove our claim, it is sufficient to show that $\mathcal{C}(\mathcal{G})$ is generated by triangles.

Suppose that $\mathcal{K}=F_1F_2\dots F_k$, with $k\ge 4$, is a shortest cycle in $\mathcal{G}$ that is not a sum of triangles in $\mathcal{C}(\mathcal{G})$. Then $\mathcal{K}$ is an induced cycle in $\mathcal{G}$, that is, $|F_i \cap F_j|\le 3$ if $2 \le |i-j| \le k-2$. Let $z\in F_1 \setminus F_2$. Then $z\in F_k$, otherwise $|F_k\cap F_1\cap F_2|=4$. Let $i$ be the smallest index such that $i\ge 3$ and $z\in F_i$. We have $i\ge 4$, otherwise $|F_1\cap F_3|=|(F_1\cap F_2\cap F_3)\cup \{z\}|=4$. For every $j\in\{2,\dots,i-2\}$, let $F'_j=(F_j\cap F_{j+1}) \cup \{z\}$. Then $\mathcal{K}$ is the sum of the closed walk $\mathcal{K}'=F_1F'_2\dots F'_{i-2}F_i \dots F_k$ and the triangles $F_1F_2F'_2,\allowbreak F_{i-1}F_iF'_{i-2}$, $F_jF_{j+1}F'_j$ for $j=2, \dots, i-2$, and $F_{j+1}F'_jF'_{j+1}$ for $j=2,\dots, i-3$; see Figure~\ref{obr_4_1_triangulace}. Since the length of $\mathcal{K}'$ is $k-1$, we have a contradiction with the choice of $\mathcal{K}$.

\begin{figure}
\begin{center}
\epsfbox{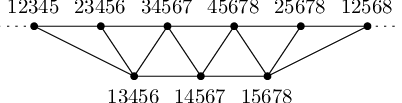}
\end{center}
\caption{Triangulating a cycle in $\mathcal{G}$. The vertices in the first row represent the vertices $F_1, \dots, F_i$ of the original cycle $\mathcal{K}$, the vertices in the second row represent the vertices $F'_2, \dots, F'_{i-2}$.}
\label{obr_4_1_triangulace}
\end{figure}

Let $Q_1Q_2Q_3$ be a nonorientable triangle in $\mathcal{G}$. The $5$-tuples $Q_1,Q_2,Q_3$ have either three or four common elements. Suppose that $|Q_1\cap Q_2\cap Q_3|=4$ and let $\{u,v,w,z\}=Q_1\cap Q_2\cap Q_3$. Then we may orient the rotation systems $\mathcal{R}(Q_1), \mathcal{R}(Q_2)$ and $\mathcal{R}(Q_3)$ as $\mathcal{R}_1$, $\mathcal{R}_2$ and $\mathcal{R}_3$, respectively, so that the rotation of $u$ in each $\mathcal{R}_i$ is compatible with $(v,w,z)$. This implies that the rotations of $u$ in $\mathcal{R}_1$, $\mathcal{R}_2$ and $\mathcal{R}_3$ are pairwise compatible. Moreover, for $i\neq j$, the rotations of $u$ in $\mathcal{R}_i$ and $(\mathcal{R}_j)^{-1}$ are not compatible. By our assumption that the rotation system $\mathcal{R}(Q_i)$ is compatible with $\mathcal{R}(Q_j)$ or its inverse, the rotation systems $\mathcal{R}_1$, $\mathcal{R}_2$ and $\mathcal{R}_3$ are pairwise compatible, a contradiction. Hence, we have $|Q_1\cap Q_2\cap Q_3|=3$, which implies that $|Q_1\cup Q_2\cup Q_3|=6$. Setting  $S=Q_1\cup Q_2\cup Q_3$, the set ${S\choose 5}$ is not orientable.
\end{proof}

If ${[n]\choose 5}$ is orientable, there are exactly two possible solutions for the orientation map. We will assume that the rotation of $1$ in $\Phi(\mathcal{R}(\{1,2,3,4,5\}))$ is compatible with $(2,3,4)$, so that there is at most one solution $\Phi$.


\subsubsection*{1c) Computing the rotations of vertices}
Having oriented the rotation system of every $5$-tuple, the algorithm now computes the rotation of every $x\in[n]$, as the cyclic permutation compatible with the rotation of $x$ in every $\Phi(\mathcal{R}(Q))$ such that $x\in Q \in {[n]\choose 5}$. We show that this is always possible.

\begin{lemma}
\label{lemma_comp_rotation}
Let $k\ge 4$. For every $F\in {[k+1]\choose k}$, let $\pi_F$ be a cyclic permutation of $F$ such that for every pair $F,F'\in {[k+1]\choose k}$, the cyclic permutations $\pi_F$ and $\pi_F'$ are compatible. Then there is a cyclic permutation $\pi_{[k+1]}$ of $[k+1]$ compatible with all the cyclic permutations $\pi_F$ with $F\in {[k+1]\choose k}$.
\end{lemma}

\begin{proof}
We may assume that $\pi_{[k]}=(1,2,\dots,k)$. For $i\in [k]$, let $F_i=[k+1]\setminus \{i\}$. Since $F_i\cap [k]=[k]\setminus \{i\}$, the cyclic permutation $\pi_{F_i}$ is compatible with $(1,2,\dots, i-i,i+1,\dots,k)$. The only freedom is thus in the position of the element $k+1$. Since $\pi_{F_1}$ and $\pi_{F_3}$ are compatible and $k\ge 4$, it is not possible that $\pi_{F_1}=(k+1,2,3,4,\dots,k)$ and $\pi_{F_3}=(1,2,k+1,4,\dots,k)$ simultaneously. Hence, in at least one of the cyclic permutations $\pi_{F_1}, \pi_{F_3}$, the position of $k+1$ is between two adjacent elements $i,i+1$ of the cycle $(1,2,\dots,k)$ (counting mod $k$). From the compatibility condition, all the other permutations $\pi_j$ are now uniquely determined and they are all compatible with the cyclic permutation $(1,2,\dots,i,k+1,i+1,\dots,k)$.
\end{proof}

Let $x\in [n]$. For every $F\in {[n]\setminus \{x\} \choose 4}$, let $\pi_F$ be the rotation of $x$ in $\Phi(\mathcal{R}(F\cup \{x\}))$. For every pair $F,F'\in {[n]\setminus \{x\} \choose 4}$ with $|F\cap F'|=3$, the cyclic permutations $\pi_F$ and $\pi_{F'}$ are compatible.
Let $k\ge 4$ and suppose that for every $F\in {[n]\setminus \{x\} \choose k}$, we have a cyclic permutation $\pi_F$ of $F$ such that for every pair $F,F'\in {[n]\setminus \{x\} \choose k}$ with $|F\cap F'|=k-1$, the cyclic permutations $\pi_F$ and $\pi_{F'}$ are compatible. By Lemma~\ref{lemma_comp_rotation}, for every $Q\in {[n]\setminus \{x\} \choose k+1}$ there is a cyclic permutation $\pi_Q$ compatible with every $\pi_F$ such that $F\in{Q\choose k}$. Moreover, for every pair $Q,Q'\in {[n]\setminus \{x\} \choose k+1}$ with $|Q\cap Q'|=k$, the cyclic permutations $\pi_Q$ and $\pi_{Q'}$ are compatible since they are both compatible with $\pi_{Q\cap Q'}$. By induction, there is a cyclic permutation of $[n]\setminus \{x\}$ that is compatible with every $\pi_F$ such that $F\in {[n]\setminus \{x\} \choose 4}$. This cyclic permutation is the rotation of $x$.


\subsubsection*{1d) Computing the rotations of crossings}
For every pair of edges $\{\{u,v\},\allowbreak \{x,y\}\}\in \mathcal{X}$, the algorithm determines the rotation of their crossing from the rotations of the vertices $u,v,x,y$. This finishes the computation of the extended rotation system.


\subsection*{Step 2: determining the homotopy classes of the edges}
Let $v$ be a fixed vertex of $A$ and let $S(v)$ be a topological star consisting of $v$ and all the edges incident with $v$, drawn in the plane so that the rotation of $v$ agrees with the rotation computed in the previous step. For every edge $e=xy$ of $A$ not incident with $v$, the algorithm computes the order of crossings of $e$ with the subset $E_{v,e}$ of edges of $S(v)$ that $e$ has to cross. By Proposition~\ref{prop_poradi_na_hvezde} (2), the five-vertex AT-subgraphs of $A$ determine the relative order of crossings of $e$ with every pair of edges of $E_{v,e}$. 
Define a binary relation $\prec_{x,y}$ on $E_{v,e}$ so that $vu \prec_{x,y} vw$ if the crossing of $e$ with $vu$ is closer to $x$ than the crossing of $e$ with $vw$. If $\prec_{x,y}$ is acyclic, it defines a total order of crossings of $e$ with the edges of $E_{v,e}$. If $\prec_{x,y}$ has a cycle, then it also has an oriented triangle $vu_1,vu_2,vu_3$. This means that the AT-subgraph of $A$ induced by the six vertices $v,u_1,u_2,u_3,x,y$ is not simply realizable. We can thus assume that $\prec_{x,y}$ is a strict total order.

We recall that the \emph{homotopy class} of a curve $\varphi$ in a surface $\Sigma$ relative to the boundary of $\Sigma$ is the set of all curves that can be obtained from $\varphi$ by a homotopy (that is, a continuous deformation) within $\Sigma$, keeping the boundary of $\Sigma$ fixed pointwise.

We define the \emph{homotopy class of $e$} using the following combinatorial data: the rotation of $v$, the set $E_{v,e}$, the total order $\prec_{x,y}$ in which the edges of $E_{v,e}$ cross $e$, the rotations of these crossings, and the rotations of the vertices $x$ and $y$. Consider the star $S(v)$ drawn on the sphere. Cut circular holes around the points representing all the vertices except $v$, and let $\Sigma$ be the resulting surface with boundary. Let $x_e$ and $y_e$ be fixed points on the boundaries of the two holes around $x$ and $y$, respectively, so that the orders of these points corresponding to all the edges of $A$ on the boundaries of the holes agree with the computed rotation system. Draw a curve $\varphi_e$ with endpoints $x_e$ and $y_e$ satisfying all the combinatorial data of $e$. Note that such a curve $\varphi_e$ always exists, but it may have to cross itself. It is also easy to see that any other curve $\varphi'_e$ drawn using the same combinatorial data is homotopic to $\varphi_e$: slightly informally, this follows from the fact that the edges of $S(v)$ together with the boundary components of $\Sigma$ form a connected subset of $\Sigma$, so the closed curve formed by the portions of $\varphi_e$ and $\varphi'_e$ between two consecutive crosings with edges $vu$ and $vw$ of $S(v)$, and by portions of the edges $vu$ and $vw$, cannot split the set of boundary components of $\Sigma$ in a nontrivial way, and hence the portion of $\varphi'_e$ can be continuously deformed onto the portion of $\varphi_e$ while keeping their endpoints on $vu$ and $vw$. Now the homotopy class of $e$ is defined as the homotopy class of $\varphi_e$ in $\Sigma$ relative to the boundary of $\Sigma$, and this is well-defined by the observation just made.


\subsection*{Step 3: computing the minimum crossing numbers}
For every pair of edges $e,f$, let ${\rm cr}(e,f)$ be the minimum possible number of crossings of two curves from the homotopy classes of $e$ and $f$. Similarly, let ${\rm cr}(e)$ be the minimum possible number of self-crossings of a curve from the homotopy class of $e$. The numbers ${\rm cr}(e,f)$ and ${\rm cr}(e)$ can be computed in polynomial time in any $2$-dimensional surface with boundary~\cite{AGR11_self,SSS08_geometric}. 
In our special case, the algorithm is relatively straightforward~\cite{K11_simple_real}.

We use the key fact that from the homotopy class of every edge, it is possible to choose a representative such that the crossing numbers ${\rm cr}(e,f)$ and ${\rm cr}(e)$ are all realized simultaneously~\cite{K11_simple_real}. This is a consequence of the following facts.

\begin{lemma}{\rm\cite{HS85_bigons}}
\label{lemma_excess1}
Let $\gamma$ be a curve on an orientable surface $S$ with endpoints on the boundary of $S$ that has more self-intersections than required by its homotopy class. Then there is a singular $1$-gon or a singular $2$-gon bounded by parts of $\gamma$. 
\end{lemma}

Here a \emph{singular $1$-gon} of a curve $\gamma:[0,1]\rightarrow S$ is an image $\gamma[\alpha]$ of an interval $\alpha\subset [0,1]$ such that $\gamma$ identifies the endpoints of $\alpha$ and the resulting loop is contractible in $S$. A \emph{singular $2$-gon} of $\gamma$ is an image of two disjoint intervals $\alpha, \beta \subset [0,1]$ such that $\gamma$ identifies the endpoints of $\alpha$ with the endpoints of $\beta$ and the resulting loop is contractible in $S$.

\begin{lemma}{\rm\cite{FT_homeomorphisms,HS85_bigons}}
\label{lemma_excess2}
Let $C_1$ and $C_2$ be two simple curves on a surface $S$ such that the endpoints of $C_1$ and $C_2$ lie on the boundary of $S$. If $C_1$ and $C_2$ have more intersections than required by their homotopy classes, then there is an innermost embedded $2$-gon between $C_1$ and $C_2$, that is, two subarcs of $C_1$ and $C_2$ bounding a disc in $S$ whose interior is disjoint from $C_1$ and $C_2$. 
\end{lemma}

Whenever there is a singular $1$-gon, a singular $2$-gon, or an innermost embedded $2$-gon in a system of curves on $S$, it is possible to eliminate the $1$-gon or $2$-gon by a homotopy of the corresponding curves, which decreases the total number of crossings.  

For the rest of the proof, we fix a drawing $D$ of $A$ such that its rotation system is the same as the rotation system computed in Step 1, the edges of $S(v)$ do not cross each other, every other edge is drawn as a curve in its homotopy class computed in Step 2, and under these conditions, the total number of crossings is the minimum possible. Then every edge $f$ of $S(v)$ crosses every other edge $e$ at most once, and this happens exactly if $\{e,f\}\in\mathcal{X}$. Moreover, for every pair of edges $e_1, e_2$ not incident with $v$, the corresponding curves in $D$ cross exactly ${\rm cr}(e_1,e_2)$ times, and the curve representing $e_1$ has ${\rm cr}(e_1)$ self-crossings. Hence, $A$ is simply realizable if and only if all the edges $e_1,e_2$ not incident with $v$ satisfy ${\rm cr}(e_1)=0$, ${\rm cr}(e_1,e_2)\le 1$, and ${\rm cr}(e_1,e_2)=1 \Leftrightarrow \{e_1,e_2\}\in\mathcal{X}$. Moreover, if $A$ is simply realizable, then $D$ is a simple realization of $A$.


\subsubsection*{3a) Characterization of the homotopy classes}

Let $w_1, w_2, \dots, w_{n-1}$ be the vertices of $A$ adjacent to $v$ and assume that the rotation of $v$ is $(w_1,w_2, \dots, w_{n-1})$. Let $w_aw_b$ be an edge such that $1\le a<b\le n-1$. Since every AT-subgraph of $A$ with four or five vertices is simply realizable, and the rotation of $v$ in every such simple realization or its mirror image is compatible with $(w_1,w_2, \dots, w_{n-1})$, we have the following conditions on the homotopy class of $w_aw_b$. 

\begin{observation}
\label{obs_zleva_zprava} 
Assume that $\{w_aw_b,vw_c\}\in\mathcal{X}$; that is, $vw_c\in E_{v,w_aw_b}$. If $a<c<b$, then the rotation of the crossing of $w_aw_b$ with $vw_c$ is $(w_a,w_c,w_b,v)$. If $c<a$ or $b<c$, then the rotation of the crossing is $(w_b,w_c,w_a,v)$.
\end{observation}

\begin{figure}
\begin{center}
\epsfbox{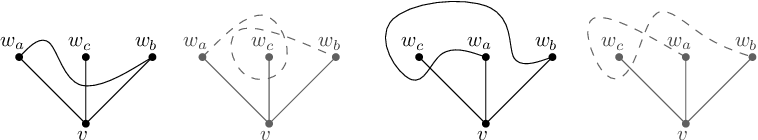}
\end{center}
\caption{Illustration to Observation~\ref{obs_zleva_zprava}. The rotation of the crossing of the edges $w_aw_b$ and $vw_c$ is determined by the rotation of $v$. The inverse rotation would force a self-crossing of the edge $w_aw_b$.}
\label{obr_4_2a_rotace_krizeni}
\end{figure}

\begin{proof}
Consider a simple realization of the complete AT-graph $A[\{v,w_a,w_b,w_c\}]$ where the rotation of $v$ is $(w_a,w_c,w_b)$ and $v$ is incident with the outer face of the drawing; see Figure~\ref{obr_4_2a_rotace_krizeni}, left. The triangle $vw_aw_b$ is then drawn as a simple closed curve, passing through the vertices $v,w_a,w_b$ in clockwise order. The portion of the curve representing the edge $vw_c$ between $v$ and the crossing with $w_aw_b$ then lies inside the triangle $vw_aw_b$, and this forces the rotation of the crossing to be $(w_a,w_c,w_b,v)$. The case when the rotation of $v$ is $(w_a,w_b,w_c)$ is symmetric; see Figure~\ref{obr_4_2a_rotace_krizeni}, right.
\end{proof}

Observation~\ref{obs_zleva_zprava} implies that the homotopy class of the edge $w_aw_b$ is determined by the total order $\prec_{w_a,w_b}$ in which $w_aw_b$ crosses the edges in $E_{v,w_aw_b}$. The next observation further restricts this total order.

\begin{observation}
\label{obs_poporade}
Assume that $vw_c,vw_d\in E_{v,w_aw_b}$. If $a<c<d<b$, then $vw_c \prec_{w_a,w_b} vw_d$.
If $(c,d,a,b)$ is compatible with $(1,2,\dots,n-1)$ as cyclic permutations, then $vw_d \prec_{w_a,w_b} vw_c$. 
\end{observation}

\begin{figure}
\begin{center}
\epsfbox{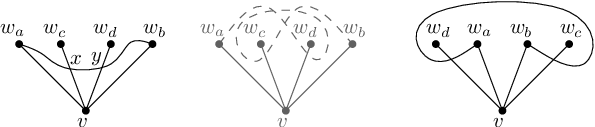}
\end{center}
\caption{Illustration to Observation~\ref{obs_poporade}. The order of the crossings of $w_aw_b$ with the edges $vw_c$ and $vw_d$ is determined by the rotation of $v$ unless $vw_a$ and $vw_b$ alternate with $vw_c$ and $vw_d$ in the rotation of $v$. The opposite order would force at least three self-crossings of the edge $w_aw_b$, assuming Observation~\ref{obs_zleva_zprava}.}
\label{obr_4_2b_poradi_krizeni}
\end{figure}

\begin{proof}
Consider a simple realization of the complete AT-graph $A[\{v,w_a,w_b,w_c,w_d\}]$ where the rotation of $v$ is $(w_a,w_c,w_d,w_b)$ and $v$ is incident with the outer face of the drawing; see Figure~\ref{obr_4_2b_poradi_krizeni}, left. Let $x$ be the crossing of the edges $w_aw_b$ and $vw_c$, and $y$ the crossing of the edges $w_aw_b$ and $vw_d$. The rotation of $v$ determines that the portion of the edge $vw_d$ between $v$ and $y$ lies inside the closed curve $vxw_b$ formed by portions of the edges $vw_c$, $w_aw_b$ and by the edge $vw_b$. This means that on the edge $w_aw_b$, the crossing $x$ is closer to $w_a$ than $y$. The case when the rotation of $v$ is $(w_a,w_b,w_c,w_d)$ is symmetric; see Figure~\ref{obr_4_2b_poradi_krizeni}, right.
\end{proof}

On the other hand, it is easy to see that every homotopy class satisfying Observations~\ref{obs_zleva_zprava} and~\ref{obs_poporade} has a representative that is a simple curve; see Figure~\ref{obr_4_2_typicka_hrana}. Therefore, $cr(w_aw_b)=0$. Let $\gamma_{a,b}$ be the simple curve representing the edge $w_aw_b$ in $D$.

\begin{figure}
\begin{center}
\epsfbox{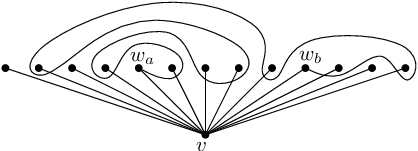}
\end{center}
\caption{A drawing of a ``typical'' edge $w_aw_b$ and the star $S(v)$.}
\label{obr_4_2_typicka_hrana}
\end{figure}



\subsubsection*{3b) The parity of the crossing numbers}

\begin{observation}
\label{obs_parity_indep}
Let $w_aw_b$ and $w_cw_d$ be two independent edges. Then ${\rm cr}(w_aw_b,\allowbreak w_cw_d)$ is odd if and only if $\{w_aw_b,\allowbreak w_cw_d\}\in\mathcal{X}$.
\end{observation}

\begin{proof}
Each of the triangles $vw_aw_b$ and $vw_cw_d$ is represented by a simple closed curve in $D$, so the total number of crossings of the two curves is even. The crossings include the vertex $v$ if $(w_a,w_c,w_b,w_d)$, $(w_a,w_d,w_b,w_c)$, or the inverse of one of these two cyclic permutations is compatible with the rotation of $v$. The following pairs of edges may cross once in $D$: $\{w_aw_b,vw_c\},\allowbreak \{w_aw_b,vw_d\},\allowbreak \{w_cw_d,vw_a\},\allowbreak \{w_cw_d,vw_b\}$. Each of these five crossings appears in $D$ if and only if it appears in a topological graph $T$ that is a simple realization of the induced AT-graph $A[\{v,w_a,w_b,w_c,w_d\}]$. The remaining crossings between the two triangles $vw_aw_b$ and $vw_cw_d$ in $D$ may appear only between the arcs $\gamma_{a,b}$ and $\gamma_{c,d}$. Since the corresponding triangles $vw_aw_b$ and $vw_cw_d$ in $T$ also have an even number of crossing in total, the parity of the number of crossings between $\gamma_{a,b}$ and $\gamma_{c,d}$ is odd if and only 
if $w_aw_b$ and $w_cw_d$ cross in $T$. The observation follows.
\end{proof}

\begin{observation}
\label{obs_parity_adj}
Let $w_aw_b$ and $w_aw_c$ be two adjacent edges. Then ${\rm cr}(w_aw_b,\allowbreak w_aw_c)$ is even.
\end{observation}

\begin{proof}
The triangle $vw_aw_b$ is represented by a simple closed curve $\gamma$ in $D$, and the path $vw_cw_a$ is represented by a simple arc $\alpha$. The rotations at $v$ and $w_a$ determine whether the end-pieces of $\alpha$ are attached to $\gamma$ from inside or outside. The following two pairs of edges can cross once in $D$: $\{w_aw_b,vw_c\}$ and $\{w_aw_c,vw_b\}$. Each of these crossings appears in $D$ if and only if it appears in a simple realization $T$ of $A[\{v,w_a,w_b,w_c\}]$. The remaining crossings between the curves $\gamma$ and $\alpha$ in $D$ may appear only between the arcs $\gamma_{a,b}$ and $\gamma_{a,c}$. The end-pieces of $\alpha$ are both attached to $\gamma$ from the same side if and only if in $T$, the end-pieces of the path $vw_cw_a$ are attached to the triangle $vw_aw_b$ from the same side. Therefore, the parity of the total number of crossings between $\gamma$ and $\alpha$ is even if and only if the path $vw_cw_a$ crosses the triangle $vw_aw_b$ an even number of times in $T$, which 
happens if and only if $T$ is planar. The parity of the number of crossings between $\gamma_{a,b}$ and $\gamma_{a,c}$ is thus equal to the parity of the number of crossings between $w_aw_b$ and $w_aw_c$ in $T$. Since $w_aw_b$ and $w_aw_c$ do not cross in $T$, the observation follows.
\end{proof}

Observations~\ref{obs_parity_indep} and~\ref{obs_parity_adj} imply that $A$ is realizable if and only if every pair of edges in $D$ crosses at most once. 


\subsubsection*{3c) Detecting crossings of adjacent edges}

We show that adjacent edges do not cross in $D$, otherwise some AT-subgraph of $A$ with five vertices is not simply realizable. Let $w_aw_b$ and $w_aw_c$ be two adjacent edges. By symmetry, we may assume that $a<b<c$. We will consider cyclic intervals $(a,b)$, $(b,c)$ and $(c,a)=(c,n-1] \cup [1,a)$. 
We define the following subsets of $E_{v,w_aw_b}$ and $E_{v,w_aw_c}$. For each of the three cyclic intervals $(i,j)$, let $F_b(i,j)=\{vw_k\in E_{v,w_aw_b}; k\in (i,j)\}$ and $F_c(i,j)=\{vw_k\in E_{v,w_aw_c}; k\in (i,j)\}$. 
We will also write $\prec_b$ as a shortcut for $\prec_{w_aw_b}$ and $\prec_c$ as a shortcut for $\prec_{w_aw_c}$.
By symmetry, we have two general cases.  


\paragraph{I) $w_aw_b$ does not cross $vw_c$ and $w_aw_c$ does not cross $vw_b$.} In this case, the rotation of $w_a$ is compatible with $(v,w_c,w_b)$. We observe the following conditions.

\begin{observation} 
\label{obs_c_a}
We have $F_b(c,a) \subseteq F_c(c,a)$ and $F_c(a,b) \subseteq F_b(a,b)$.
\end{observation}

\begin{proof}
We prove the first inclusion, the second one follows by symmetry. Let $d\in (c,a)$ such that $w_aw_b$ crosses $vw_d$. Let $T$ be a simple realization of $A[\{v,\allowbreak w_a,\allowbreak w_b,\allowbreak w_c,\allowbreak w_d\}]$. Refer to Figure~\ref{obr_4_3_I}, left. The vertices $w_a$ and $w_c$ are separated by the simple closed curve formed by the edge $vw_b$ and portions of the edges $w_bw_a$ and $vw_d$. Since $w_aw_c$ cannot cross $w_aw_b$ in $T$, and by the assumption of case I) it does not cross $vw_b$ either, the edge $w_aw_c$ is forced to cross $vw_d$ in $T$.
\end{proof}

\begin{figure}
\begin{center}
\epsfbox{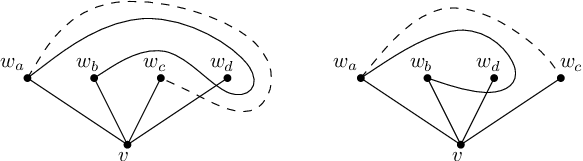}
\end{center}
\caption{Left: illustration to Observation~\ref{obs_c_a}. Right: illustration to Observation~\ref{obs_b_c_disjoint}.}
\label{obr_4_3_I}
\end{figure}

\begin{observation} 
\label{obs_b_c_disjoint}
The sets $F_b(b,c)$ and $F_c(b,c)$ are disjoint.
\end{observation}

\begin{proof}
Let $d\in (b,c)$ such that $w_aw_b$ crosses $vw_d$. Let $T$ be a simple realization of $A[\{v,\allowbreak w_a,\allowbreak w_b,\allowbreak w_c,\allowbreak w_d\}]$ such that $v$ is incident with the outer face. Refer to Figure~\ref{obr_4_3_I}, right. Since $w_aw_c$ is adjacent to $vw_a$ and $w_aw_b$, it cannot cross the triangle $vw_aw_b$, and so $w_aw_c$ must be drawn outside the triangle $vw_aw_b$. Therefore, $w_aw_c$ cannot cross $vw_d$ in $T$, otherwise it would have to cross $vw_d$ at least twice.
\end{proof}

\begin{observation} 
\label{obs_b_c_separated}
If $vw_d\in F_b(b,c)$ and $vw_e\in F_c(b,c)$, then $d<e$.
\end{observation}

\begin{proof}
Let $d,e\in(b,c)$ such that $vw_d\in F_b(b,c)$ and $vw_e\in F_c(b,c)$. Observation~\ref{obs_b_c_disjoint} implies that $d\neq e$. We claim that if $e<d$ then $A[\{v,\allowbreak w_a,\allowbreak w_b,\allowbreak w_c,\allowbreak w_d,\allowbreak w_e\}]$ is not simply realizable. Suppose, for a contradiction, that $T$ is a simple realization of $A[\{v,\allowbreak w_a,\allowbreak w_b,\allowbreak w_c,\allowbreak w_d,\allowbreak w_e\}]$ where $v$ is incident with the outer face and the rotation of $v$ is $(w_a,\allowbreak w_b,\allowbreak w_e,\allowbreak w_d,\allowbreak w_c)$. See Figure~\ref{obr_4_4_I}, left. By Observation~\ref{obs_b_c_disjoint}, the edges $w_aw_b$ and $vw_e$ do not cross in $T$. The edge $w_aw_d$ is drawn inside the triangle $vw_aw_b$ and the edge $vw_e$ is drawn outside the triangle $vw_aw_b$, thus $w_aw_d$ does not cross $vw_e$. By Observation~\ref{obs_b_c_disjoint} the edges $w_aw_c$ and $vw_d$ do not cross, and so the edge $w_aw_d$ is drawn inside the triangle $vw_aw_c$. Within this triangle, a portion of the edge $vw_e$ separates $w_a$ from $w_d$, and therefore $w_aw_d$ and $vw_e$ are forced to cross; a contradiction.
\end{proof}

\begin{figure}
\begin{center}
\epsfbox{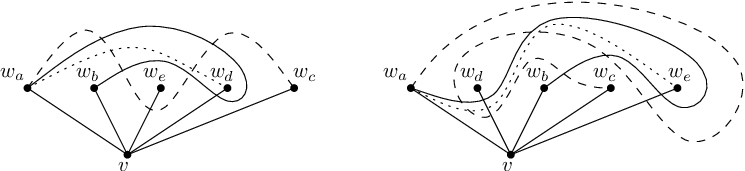}
\end{center}
\caption{Left: illustration to Observation~\ref{obs_b_c_separated}. Right: illustration to Observation~\ref{obs_a_b__c_a}.}
\label{obr_4_4_I}
\end{figure}

\begin{observation} 
\label{obs_a_b__c_a}
Let $vw_d\in F_b(a,b)\cap F_c(a,b)$ and $vw_e\in F_b(c,a)\cap F_c(c,a)$. Then $vw_d \prec_b vw_e \Leftrightarrow vw_d \prec_c vw_e$. 
\end{observation}

\begin{proof}
Suppose that $vw_d \prec_b vw_e$ and $vw_e \prec_c vw_d$. We claim that $A[\{v,\allowbreak w_a,\allowbreak w_b,\allowbreak w_c,\allowbreak w_d,\allowbreak w_e\}]$ is not simply realizable.  Suppose, for a contradiction, that $T$ is a simple realization of $A[\{v,\allowbreak w_a,\allowbreak w_b,\allowbreak w_c,\allowbreak w_d,\allowbreak w_e\}]$ where $v$ is incident with the outer face and the rotation of $v$ is $(w_a,\allowbreak w_d,\allowbreak w_b,\allowbreak w_c,\allowbreak w_e)$. See Figure~\ref{obr_4_4_I}, right. The edge $w_aw_e$ must be drawn inside the triangle $vw_aw_b$ in $T$. Since a portion of $vw_d$ separates $w_a$ from $w_e$ within this triangle, the edges $w_aw_e$ and $vw_d$ cross. 
The edge $w_aw_e$ is also drawn inside the triangle $vw_aw_c$. The edge $vw_d$ separates the triangle $vw_aw_c$ into two regions but does not separate $w_a$ from $w_e$, and therefore $w_aw_e$ and $vw_d$ cannot cross; a contradiction.
\end{proof}

\begin{observation} 
\label{obs_a_b__b_c}
Let $vw_d\in F_c(a,b)$ and $vw_e\in F_b(b,c)$. Then $vw_d \prec_b vw_e$. Similarly, if $vw_d\in F_b(c,a)$ and $vw_e\in F_c(b,c)$, then $vw_d \prec_c vw_e$.
\end{observation}

\begin{proof}
We show the first part, the second part follows by symmetry. Observation~\ref{obs_c_a} implies that $vw_d\in F_b(a,b)$ and Observation~\ref{obs_b_c_disjoint} implies that $vw_e\notin F_c(b,c)$. Suppose that $vw_e \prec_b vw_d$. We claim that $A[\{v,\allowbreak w_a,\allowbreak w_b,\allowbreak w_c,\allowbreak w_d,\allowbreak w_e\}]$ is not simply realizable. Suppose, for a contradiction, that $T$ is a simple realization of $A[\{v,\allowbreak w_a,\allowbreak w_b,\allowbreak w_c,\allowbreak w_d,\allowbreak w_e\}]$  where $v$ is incident with the outer face and the rotation of $v$ is $(w_a,\allowbreak w_d,\allowbreak w_b,\allowbreak w_e,\allowbreak w_c)$. See Figure~\ref{obr_4_5_I}, left. The edge $w_aw_e$ is drawn inside the triangle $vw_aw_b$ in $T$. The edge $vw_d$ separates the triangle $vw_aw_b$ into two regions but does not separate $w_a$ from $w_e$, thus $w_aw_e$ and $vw_d$ do not cross.
The edge $w_aw_e$ is also drawn inside the triangle $vw_aw_c$. Within this triangle, the edge $vw_d$ separates $w_a$ from $w_e$, and therefore $w_aw_e$ and $vw_d$ cross; a contradiction.
\end{proof}

\begin{figure}
\begin{center}
\epsfbox{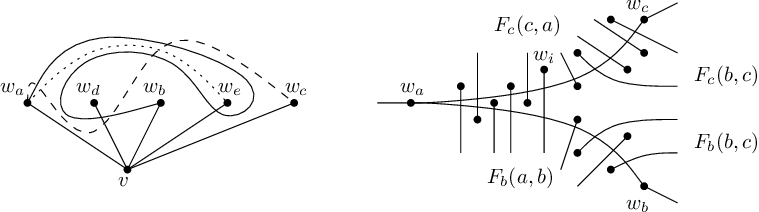}
\end{center}
\caption{Left: Illustration to Observation~\ref{obs_a_b__b_c}. Right: A drawing of the edges $w_aw_b$, $w_aw_c$ and parts of edges of $S(v)$ in case I), where $w_aw_b$ and $w_aw_c$ do not cross. The rotation of $v$ is compatible with the counterclockwise cyclic order of the parts of the edges drawn.}
\label{obr_4_5_I}
\end{figure}

Assuming Observations~\ref{obs_zleva_zprava} and~\ref{obs_poporade}, we show that Observations~\ref{obs_c_a}--\ref{obs_a_b__b_c} imply that ${\rm cr}(w_aw_b,\allowbreak w_aw_c)=0$. Refer to Figure~\ref{obr_4_5_I}, right. Start with drawing the edges $w_aw_b$ and $w_aw_c$ without crossing.
Observations~\ref{obs_b_c_disjoint} and~\ref{obs_a_b__c_a} imply that there is a total order $\prec$ on $E_{v,w_aw_b} \cup E_{v,w_aw_c}$ that is a common extension of $\prec_b$ and $\prec_c$. Let $vw_i$ be the $\prec$-largest element of $F_b(c,a) \cup F_c(a,b)$. Observation~\ref{obs_a_b__b_c} implies that all edges $vw_j$ from $F_b(b,c) \cup F_c(b,c)$ satisfy $vw_i \prec vw_j$. Observation~\ref{obs_c_a} implies that we can draw the edges $vw_j$ with $vw_j \preceq vw_i$ like in the figure. Observations~\ref{obs_b_c_disjoint},~\ref{obs_b_c_separated} and~\ref{obs_a_b__b_c}, imply that we can draw the edges $vw_j$ with $vw_i \prec vw_j$ like in the figure. The remaining edges of $S(v)$ can be drawn easily. In this way we obtain a simple drawing with noncrossing representatives of the homotopy classes of $w_aw_b$ and $w_aw_c$.


\paragraph{II) $w_aw_b$ does not cross $vw_c$ and $w_aw_c$ crosses $vw_b$.} In this case, the rotation of $w_a$ is compatible with $(v,w_b,w_c)$. Let $x$ be the crossing of $w_aw_c$ and $vw_b$. For each of the three cyclic intervals $(i,j)$ with endpoints $a,b$ or $c$, let $F_{w_ax}(i,j)$ be the set of edges $vw_k$ such that the part $w_ax$ of $w_aw_c$ crosses $vw_k$ and $k\in (i,j)$. Similarly, let $F_{xw_c}(i,j)$ be the set of edges $vw_k$ such that the part $xw_c$ of $w_aw_c$ crosses $vw_k$ and $k\in (i,j)$. Clearly, $F_c(i,j)=F_{w_ax}(i,j) \cup F_{xw_c}(i,j)$. We observe the following conditions.

\begin{observation} 
\label{obs_II_a_b}
We have $F_b(a,b) \subseteq F_{w_ax}(a,b)$.
\end{observation} 

\begin{proof}
Let $d\in (a,b)$ such that $w_aw_b$ crosses $vw_d$. Let $T$ be a simple realization of $A[\{v,\allowbreak w_a,\allowbreak w_b,\allowbreak w_c,\allowbreak w_d\}]$ where $v$ is incident with the outer face. See Figure~\ref{obr_4_6_II}, left. Since the curve $w_ax$ lies inside the triangle $vw_aw_b$ and $vw_d$ separates $w_a$ from $x$ within this triangle, the curve $w_ax$ is forced to cross $vw_d$ in $T$.
\end{proof}

\begin{figure}
\begin{center}
\epsfbox{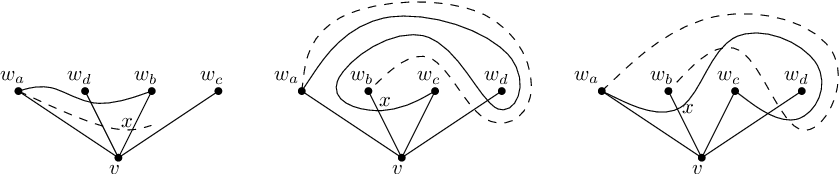}
\end{center}
\caption{Left: Illustration to Observation~\ref{obs_II_a_b}. Middle: Illustration to Observation~\ref{obs_II_c_a}. Right: Illustration to Observation~\ref{obs_II_c_a__x_w_c}.}
\label{obr_4_6_II}
\end{figure}

\begin{observation} 
\label{obs_II_c_a}
We have $F_{w_ax}(c,a) \subseteq F_b(c,a)$.
\end{observation} 

\begin{proof}
Let $d\in (c,a)$ such that $w_ax$ crosses $vw_d$. Let $T$ be a simple realization of $A[\{v,\allowbreak w_a,\allowbreak w_b,\allowbreak w_c,\allowbreak w_d\}]$ where $v$ is incident with the outer face. See Figure~\ref{obr_4_6_II}, middle. Since $w_aw_b$ cannot cross the triangle $vw_aw_c$, and the edge $vw_d$ separates $w_a$ from $w_b$ outside the triangle, the edge $w_aw_b$ is forced to cross $vw_d$ in $T$.
\end{proof}

\begin{observation} 
\label{obs_II_c_a__x_w_c}
Let $d,e\in(c,a)$ such that $vw_d\in F_{xw_c}(c,a)$ and $vw_e\in F_{b}(c,a)$. Then $d\neq e$ and $(d,e,a)$ is compatible with $(1,2\dots,n-1)$.
\end{observation} 

\begin{proof}
First we show that $d\neq e$. Suppose that $T$ is a simple realization of $A[\{v,\allowbreak w_a,\allowbreak w_b,\allowbreak w_c,\allowbreak w_d\}]$ where $v$ is incident with the outer face. See Figure~\ref{obr_4_6_II}, right. Since $w_b$ is outside the triangle $vw_aw_c$, and the edge $vw_d$ separates the region outside the triangle into two regions but does not separate $w_a$ from $w_b$, the edge $w_aw_b$ cannot cross $vw_d$. Therefore $d\neq e$.

Suppose that $(e,d,a)$ is compatible with $(1,2\dots,n-1)$. By the first part we may assume that $vw_d\notin F_{b}(c,a)$ and $vw_e\notin F_{xw_c}(c,a)$. We claim that $A[\{v,\allowbreak w_a,\allowbreak w_b,\allowbreak w_c,\allowbreak w_d,\allowbreak w_e\}]$ is not simply realizable. Suppose, for a contradiction, that $T$ is a simple realization of $A[\{v,\allowbreak w_a,\allowbreak w_b,\allowbreak w_c,\allowbreak w_d,\allowbreak w_e\}]$ where $v$ is incident with the outer face and the rotation of $v$ is $(w_a,\allowbreak w_b,\allowbreak w_c,\allowbreak w_e,\allowbreak w_d)$. See Figure~\ref{obr_4_7_II}, left. Since $w_aw_d$ is inside the triangle $vw_aw_c$ in $T$ and $w_e$ outside the triangle, the edges $w_aw_d$ and $vw_e$ do not cross. 
The curve $xw_d$ must be drawn outside the triangle $vw_aw_b$. However, the edge $vw_e$ separates $x$ from $w_d$ outside the triangle, thus $w_aw_d$ and $vw_e$ are forced to cross; a contradiction.
\end{proof}

\begin{figure}
\begin{center}
\epsfbox{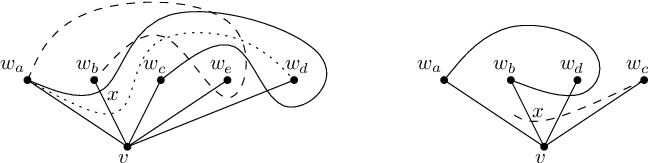}
\end{center}
\caption{Left: Illustration to Observation~\ref{obs_II_c_a__x_w_c}. The edge $w_aw_b$ is forced to cross $vw_d$ and hence also $w_aw_c$. Right: Illustration to Observation~\ref{obs_II_b_c}.}
\label{obr_4_7_II}
\end{figure}

\begin{observation} 
\label{obs_II_b_c}
We have $F_b(b,c) \subseteq F_{xw_c}(b,c)$.
\end{observation} 

\begin{proof}
Let $d\in (b,c)$ such that $w_aw_b$ crosses $vw_d$. Let $T$ be a simple realization of $A[\{v,\allowbreak w_a,\allowbreak w_b,\allowbreak w_c,\allowbreak w_d\}]$ where $v$ is incident with the outer face. See Figure~\ref{obr_4_7_II}, right. The curve $xw_c$ must be drawn outside the triangle $vw_aw_b$ in $T$. However, the edge $vw_d$ separates $x$ from $w_c$ outside the triangle, thus $xw_c$ and $vw_d$ are forced to cross in $T$.
\end{proof}

\begin{observation} 
\label{obs_II_a_b__c_a_poradi}
Let $vw_d\in F_b(a,b)\cap F_{w_ax}(a,b)$ and $vw_e\in F_b(c,a)\cap F_{w_ax}(c,a)$. Then $vw_d \prec_b vw_e \Leftrightarrow vw_d \prec_c vw_e$. 
\end{observation}

\begin{proof}
Suppose, for a contradiction, that $vw_e \prec_b vw_d$ and $vw_d \prec_c vw_e$. Let $T$ be a simple realization of $A[\{v,\allowbreak w_a,\allowbreak w_b,\allowbreak w_c,\allowbreak w_d,\allowbreak w_e\}]$ where $v$ is incident with the outer face. See Figure~\ref{obr_4_8_II}, left. The edge $w_aw_e$ is drawn inside the triangle $vw_aw_b$ in $T$. The edge $vw_d$ separates the triangle $vw_aw_b$ into two regions but does not separate $w_a$ from $w_e$, thus the edges $w_aw_e$ and $vw_d$ do not cross. The edge $w_aw_e$ is also drawn inside the triangle $vw_aw_c$. The edge $vw_d$ separates $w_a$ from $w_e$ within the triangle $vw_aw_c$, and hence $w_aw_e$ and $vw_d$ are forced to cross; a contradiction.

Suppose, for a contradiction, that $vw_d \prec_b vw_e$ and $vw_e \prec_c vw_d$. Let $T$ be a simple realization of $A[\{v,\allowbreak w_a,\allowbreak w_b,\allowbreak w_c,\allowbreak w_d,\allowbreak w_e\}]$ where $v$ is incident with the outer face. See Figure~\ref{obr_4_8_II}, right. The edge $w_aw_e$ is drawn inside the triangle $vw_aw_b$ in $T$. The edge $vw_d$ separates $w_a$ from $w_e$ within the triangle $vw_aw_b$, and hence $w_aw_e$ and $vw_d$ cross. The edge $w_aw_e$ is also drawn inside the triangle $vw_aw_c$. The edge $vw_d$ separates the triangle $vw_aw_c$ into two regions but does not separate $w_a$ from $w_e$, thus the edges $w_aw_e$ and $vw_d$ cannot cross; a contradiction.
\end{proof}

\begin{figure}
\begin{center}
\epsfbox{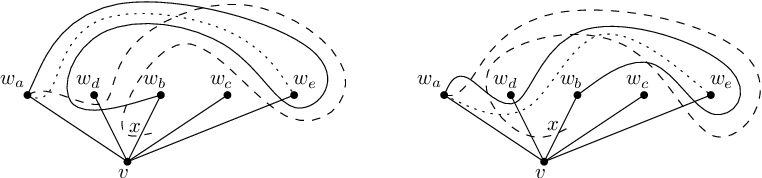}
\end{center}
\caption{Illustrations to Observation~\ref{obs_II_a_b__c_a_poradi}.}
\label{obr_4_8_II}
\end{figure}

\begin{observation} 
\label{obs_II_b_c__xw_c_poradi}
Let $vw_d\in F_b(b,c)\cap F_{xw_c}(b,c)$ and $vw_e\in F_{xw_c}(c,a)$. Then $vw_d \prec_c vw_e$.
\end{observation}

\begin{proof}
Suppose, for a contradiction, that $vw_e \prec_c vw_d$. Let $T$ be a simple realization of $A[\{v,\allowbreak w_a,\allowbreak w_b,\allowbreak w_c,\allowbreak w_d,\allowbreak w_e\}]$ where $v$ is incident with the outer face. See Figure~\ref{obr_4_9_II}, left. The edge $w_aw_d$ is drawn inside the triangle $vw_aw_b$ in $T$ and the edge $vw_e$ is drawn outside of the triangle $vw_aw_b$, thus the edges $w_aw_d$ and $vw_e$ do not cross. The edge $w_aw_d$ is also drawn outside the triangle $vw_aw_c$. Since the edge $vw_e$ separates $w_a$ from $w_d$ outside the triangle, the edges $w_aw_d$ and $vw_e$ are forced to cross; a contradiction.
\end{proof}

\begin{figure}
\begin{center}
\epsfbox{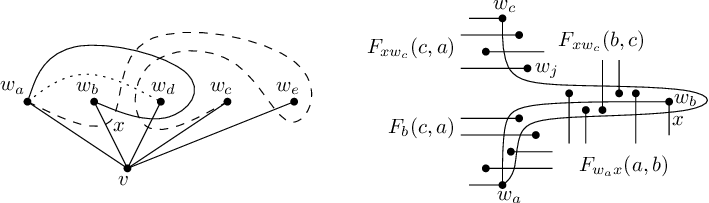}
\end{center}
\caption{Left: Illustration to Observation~\ref{obs_II_b_c__xw_c_poradi}. Right: A drawing of the edges $w_aw_b$, $w_aw_c$ and parts of the edges of $S(v)$ in case II), where $w_aw_b$ and $w_aw_c$ do not cross. The rotation of $v$ is compatible with the counterclockwise cyclic order of the parts of the edges drawn.}
\label{obr_4_9_II}
\end{figure}

Assuming Observations~\ref{obs_zleva_zprava} and~\ref{obs_poporade}, we show that Observations~\ref{obs_II_a_b}--\ref{obs_II_b_c__xw_c_poradi} imply that ${\rm cr}(w_aw_b,\allowbreak w_aw_c)=0$. Refer to Figure~\ref{obr_4_9_II}, right. Start with drawing the edges $w_aw_b$ and $w_aw_c$ without crossing and so that $w_aw_c$ crosses $vw_b$ from the left; that is, so that the rotation of the crossing $x$ is $(w_a,w_b,w_c,v)$.

Observation~\ref{obs_II_a_b__c_a_poradi} implies that there is a total order $\prec'$ on $F_{w_ax}(a,b) \cup F_{w_ax}(c,a) \cup E_{v,w_aw_b}$ that is a common extension of $\prec_b$ and $\prec_c$. 
By Observation~\ref{obs_poporade}, all pairs of edges $vw_d,vw_e \in F_b(b,c)$ satisfy $vw_d \prec_b vw_e \Leftrightarrow vw_e \prec_c vw_d$, thus
there is also a total order $\prec''$ on $F_{xw_c}(b,c) \cup F_{xw_c}(c,a) \cup E_{v,w_aw_b}$ that is a common extension of $\prec_c$ and the inverse of $\prec_b$. 
Let $vw_j$ be the $\prec''$-smallest element of $F_{xw_c}(c,a)$.
Observations~\ref{obs_II_a_b} and~\ref{obs_II_c_a} imply that we can draw the edges of $F_{w_ax}(a,b) \cup F_{w_ax}(c,a) \cup F_b(a,b) \cup F_b(c,a)$ like in the figure. Observation~\ref{obs_II_b_c} implies that we can draw the edges of $F_{xw_c}(b,c) \cup F_b(b,c)$ like in the figure. 
Observation~\ref{obs_II_c_a__x_w_c} implies that the edges of $F_b(c,a)$ are separated from the edges of $F_{xw_c}(c,a)$ in the rotation of $v$, and Observation~\ref{obs_II_b_c__xw_c_poradi} implies that all edges $vw_k\in F_b(b,c)$ satisfy $vw_k \prec'' vw_j$, therefore we can draw the edges of $F_{xw_c}(c,a)$ like in the figure.


\subsubsection*{3d) Detecting multiple crossings of independent edges}

There are several possible approaches to finding small obstructions in the case when independent edges cross more than once. By the analysis in~\cite{K11_simple_real}, if two independent edges $e,f$ cross at least twice, then a subset of at most $13$ vertices is responsible: this includes the vertex $v$, the endpoints of $e$ and $f$, and at most four more vertices for each crossing.

We could proceed analogously as in the previous substep, by investigating drawings of a graph consisting of the star $S(v)$ and two additional independent edges $e,f$. We would show that if the homotopy classes of $e,f$ force the two edges to cross at least twice, then it is typically due to a subgraph with six or seven vertices, including $v$, the endpoints of $e$ and $f$, and one or two more vertices preventing simplification of the drawing; see Figure~\ref{obr_4_ind_0_example} for one such example. However, this would lead to a rather tedious case analysis.

\begin{figure}
\begin{center}
\epsfbox{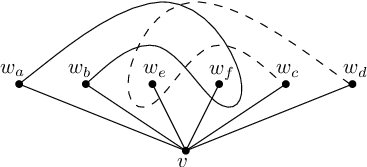}
\end{center}
\caption{Two independent edges crossing more than once in a subgraph with seven vertices.} 
\label{obr_4_ind_0_example}
\end{figure}

Therefore, we choose another approach. We use the fact that $D$ is a drawing of a complete graph and that adjacent edges do not cross in $D$, which is guaranteed by the results of the previous step. We proceed by induction. Whenever we have two independent edges that cross 
at least twice, we either find a drawing of a forbidden subgraph with just five vertices, or another pair of independent edges that cross at least twice, but in a smaller subgraph.
Now we make the idea more precise.

Let $e=w_aw_b$ and $f=w_cw_d$ be two independent edges that cross more than once in $D$. 
In the subgraph of $D$ formed by the two edges $e$ and $f$, the vertices $w_a,w_b,w_c,w_d$ are incident to a common face, since adjacent edges do not cross in $D$ and every pair of the four vertices $w_a,w_b,w_c,w_d$ is connected by an edge. We assume without loss of generality that $w_a,w_b,w_c,w_d$ are incident to the outer face. That is, we may draw a simple closed curve $\gamma$ containing the vertices $w_a,w_b,w_c,w_d$ but no interior points of $e$ or $f$, such that the relative interiors of $e$ and $f$ are inside $\gamma$. We distinguish two cases according to the parity of the number of crossings of $e$ and $f$.


\paragraph*{I) The edges $e$ and $f$ cross an even number of times.}
%
The edge $e$ splits the region inside $\gamma$ into two regions, $R_0(e)$ and $R_1(e)$, where $R_0(e)$ is the region 
that does not contain the endpoints of $f$ on its boundary. Similarly, $f$ splits the region inside $\gamma$ into regions $R_0(f)$ and $R_1(f)$ where $R_0(f)$ is the region 
that does not contain the endpoints of $e$ on its boundary.

\begin{figure}
\begin{center}
\epsfbox{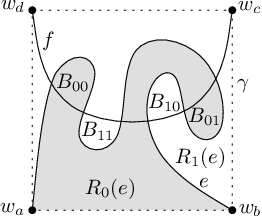}
\end{center}
\caption{Four types of bigons between $e$ and $f$. An $ij$-bigon is denoted by $B_{ij}$.} 
\label{obr_4_ind_1_regions}
\end{figure}

By Lemma~\ref{lemma_excess2}, there is an innermost embedded $2$-gon between $e$ and $f$. For brevity, we call an innermost embedded $2$-gon shortly a \emph{bigon}. For a bigon $B$, by $B^o$ we denote the open region inside $B$ and we call it the \emph{inside of $B$}. There are four possible types of bigons between $e$ and $f$, according to the regions $R_i(e)$ and $R_j(f)$ in which their insides are contained. For $i,j\in \{0,1\}$, we call a bigon $B$ an \emph{$ij$-bigon} if $B^o\subseteq R_i(e) \cap R_j(f)$; see Figure~\ref{obr_4_ind_1_regions}.

Since $D$ is a drawing realizing the crossing number ${\rm cr}(e,f)$, there is at least one vertex of $D$ inside every bigon. The graph induced by $v$, the endpoints of $e$ and $f$, and a set of vertices intersecting all bigons, certifies that $e$ and $f$ have at least ${\rm cr}(e,f)$ crossings forced by their homotopy classes. 

We use the following classical result about drawings of $K_5$ in the plane. See also Observation~\ref{obs_K5_2K3} for a similar statement and its proof.

\begin{lemma}[{Kleitman~\cite{Kle76_parity}}]
\label{lemma_parity_K5_K3_3}
In every drawing of $K_5$ in the plane the total number of pairs of independent edges crossing an odd number of times is odd. In particular, in every simple drawing of  $K_5$ in the plane the total number of crossings is odd.
\end{lemma}

\begin{observation}
\label{obs_no4odd}
Let $D'$ be a drawing of $K_5$ that is a subdrawing of $D$. For every vertex $v$ of $D'$, there is an edge $vw$ in $D'$ that crosses each of the other edges in $D'$ an even number of times.
\end{observation}

\begin{proof}
By Observations~\ref{obs_parity_indep} and~\ref{obs_parity_adj}, it is sufficient to show that in simple drawings of $K_5$, it is not possible that for some vertex $v$ all four edges incident with $v$ have a crossing. If there was such a vertex, there would have to be five crossings in total by Lemma~\ref{lemma_parity_K5_K3_3}. Up to isomorphism, there are two different simple drawings of $K_5$ with five crossings~\cite{HM92_drawings}; see, for example, the first two drawings in Figure~\ref{obr_2_realizations}. Clearly, all vertices in these two drawings satisfy the claim.
\end{proof}

\begin{lemma}
If $e$ and $f$ cross evenly, then there is no $00$-bigon between $e$ and $f$ in $D$.
\end{lemma}

\begin{proof}
Let $B$ be a $00$-bigon between $e$ and $f$ in $D$. Then there is a vertex $w_i$ inside $B$.
Consider a subdrawing $D'$ of $D$ formed by the edges $e,f,w_aw_c$ and $w_bw_c$; see Figure~\ref{obr_4_ind_2_00bigon}. Since adjacent edges do not cross in $D$, the only crossings in $D'$ are those between $e$ and $f$. In particular, the triangle $w_aw_bw_c$ is a simple closed curve; call it $\gamma_{abc}$. Since $w_i\in R_0(e)$, $w_d$ is outside $R_0(e)$, and the edge $f$ crosses the curve $\gamma_{abc}$ only on $e$, it follows that $w_i$ and $w_d$ are on opposite sides of $\gamma_{abc}$. Therefore, the edge $w_iw_d$ has to cross some edge of the triangle $w_aw_bw_c$ an odd number of times. 
By symmetry, in the subgraph induced by the $5$-tuple $\{w_a,w_b,w_c,w_d,w_i\}$, every edge adjacent to $w_i$ crosses some other edge an odd number of times. This contradicts Observation~\ref{obs_no4odd}.
\end{proof}

\begin{figure}
\begin{center}
\epsfbox{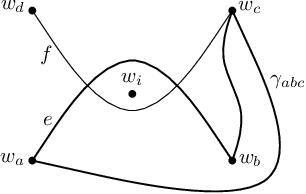}
\end{center}
\caption{The vertices $w_i$ and $w_d$ are on opposite sides of the curve $\gamma_{abc}$.} 
\label{obr_4_ind_2_00bigon}
\end{figure}

\begin{observation}
\label{obs_01_10}
If $e$ and $f$ cross evenly and at least twice in $D$, and there is no $00$-bigon between $e$ and $f$, then there is a $01$-bigon and a $10$-bigon between $e$ and $f$.
\end{observation}

\begin{proof}
If there is no $00$-bigon and no $01$-bigon between $e$ and $f$, then there is no bigon in the region $R_0(e)$, so $e$ and $f$ have no crossings.
\end{proof}

Now we define crossing numbers of edges \emph{relative} to a subset of vertices. For a subset $W$ of vertices of $A$ containing $v$ and the endpoints of two edges $e$ and $f$, let ${\rm cr}_W(e,f)$ be the minimum possible number of crossings of two curves from the homotopy classes of $e$ and $f$ determined by $A[W]$, by a procedure analogous to the one in Step 2. That is, if we start with the subdrawing of $D$ consisting of the vertices of $W$ and the edges $e$ and $f$, and simplify it sequentially by removing all bigons with no vertices of $W$ inside them, we get a drawing where $e$ and $f$ cross ${\rm cr}_W(e,f)$ times.

The following lemma proves the induction step in the case when $e$ and $f$ cross an even number of times.

\begin{lemma}
If $e$ and $f$ cross evenly and at least twice in $D$, and there is no $00$-bigon between $e$ and $f$, then there is 
a proper subset $W$ of vertices of $A$ and an edge $g$ independent from $f$ such that ${\rm cr}_W(g,f) \ge 2$.
\end{lemma}

\begin{proof}
By Observation~\ref{obs_01_10}, there is a $01$-bigon $B$ between $e$ and $f$. Let $w_i$ be a vertex inside $B$. 
Let $W'$ be a minimal subset of vertices containing $v,\allowbreak w_a,\allowbreak w_b,\allowbreak w_c,\allowbreak w_d,\allowbreak w_i$ and at least one vertex inside every bigon between $e$ and $f$ (except $B$). 
By the definition of $W'$, we have $\mathrm{cr}_{W'}(e,f)=\mathrm{cr}(e,f)$. We may assume that $W'$ is the vertex set of $A$, otherwise we just take $g=e$. The edges $w_aw_i$ and $w_bw_i$ are adjacent to $e$ and thus do not cross $e$. Moreover, in the drawing induced by $e$ and $f$, the vertices $w_a$ and $w_b$ are separated from $w_i$ by at least two parts of $f$ connecting points on $e$. This implies that $\mathrm{cr}_{W'}(w_aw_i,f)\ge 2$ and $\mathrm{cr}_{W'}(w_bw_i,f)\ge 2$.

We claim that at least one of the equalities 
\[
\mathrm{cr}_{W'\setminus \{w_b\}}(w_aw_i,f)=\mathrm{cr}_{W'}(w_aw_i,f) \ \text{ or }\ \mathrm{cr}_{W'\setminus \{w_a\}}(w_bw_i,f)=\mathrm{cr}_{W'}(w_bw_i,f)
\]
is satisfied; this will imply the lemma by taking $W=W'\setminus \{w_b\}$ or $W=W'\setminus \{w_a\}$. Suppose the contrary. Then by Lemma~\ref{lemma_excess2}, the vertex $w_b$ is inside a bigon $B_b$ between a part of $w_aw_i$ and $f$, and $w_a$ is inside a bigon $B_a$ between a part of $w_bw_i$ and $f$; see Figure~\ref{obr_4_ind_3_3bigons}.
Note that some edges of the triangle $w_aw_bw_i$ can intersect the inside of $B_a$ or $B_b$. However, $w_a$ and $w_i$ are not inside $B_b$, since they are connected to $w_c$ and $w_d$ by edges adjacent to $w_aw_i$ and $f$. Similarly, $w_b$ and $w_i$ are not inside $B_a$. Thus, $B$, $B_a$ and $B_b$ are the smallest bigons between $f$ and a part of an edge of the triangle $w_aw_bw_i$ that contain $w_i$, $w_a$ and $w_b$, respectively. Since the edges $e$, $w_aw_i$ and $w_bw_i$ are pairwise adjacent, 
they do not cross, and this implies that the interiors of the bigons $B$, $B_a$ and $B_b$ are pairwise disjoint; their boundaries may share a portion of $f$ though. Let $x_i$ be a point on $e$ on the boundary of $B$, $x_a$ a point on $w_bw_i$ on the boundary of $B_a$, and $x_b$ a point on $w_aw_i$ on the boundary of $B_b$. By joining $x_i$ with $w_i$ inside $B$, $x_a$ with $w_a$ inside $B_a$, and $x_b$ with $w_b$ inside $B_b$, we obtain a plane drawing of $K_{3,3}$; a contradiction.
\end{proof}

\begin{figure}
\begin{center}
\epsfbox{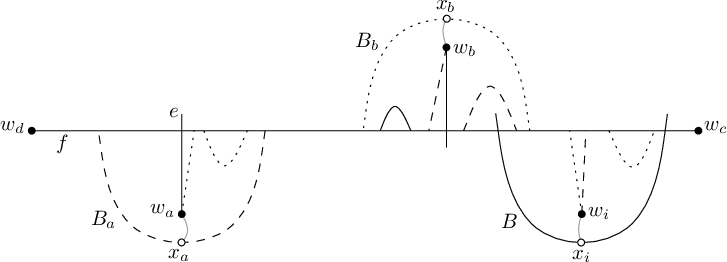}
\end{center}
\caption{Three interior-disjoint bigons between $f$ and edges of the triangle $w_aw_bw_i$ yield a plane drawing of $K_{3,3}$.} 
\label{obr_4_ind_3_3bigons}
\end{figure}


\paragraph*{II) The edges $e$ and $f$ cross an odd number of times.}
In this case, we find another pair of edges crossing evenly.

\begin{lemma}
If $e$ and $f$ cross oddly and at least three times in $D$, then there are two independent edges $e',f'$ such that ${\rm cr}(e',f')$ is even and ${\rm cr}(e',f') \ge 2$.
\end{lemma}

\begin{proof}
By Lemma~\ref{lemma_excess2}, there is a bigon $B$ between $e$ and $f$. Let $w_i$ be a vertex inside $B$. Every edge connecting $w_i$ with an endpoint of $e$ or $f$ crosses $f$ or $e$ at least once. By Observation~\ref{obs_no4odd}, at least one of these four edges crosses each of $e$ and $f$ an even number of times.
\end{proof}

\section{Independent $\mathbb{Z}_2$-realizability}
\label{section_Z2}

\subsection{Algebraic algorithm}

By the Hanani--Tutte theorem, planarity testing can be reduced to solving a system of linear equations over $\mathbb{Z}_2$~\cite[Section 1.4.1]{Sch13_hananitutte}. The same algebraic method can be used for testing independent $\mathbb{Z}_2$-realizability of general AT-graphs. Pach and T\'oth described a slightly more general method for proving the NP-completeness of computing the odd crossing number~\cite{PT00_which}, where they also took the rotation system into account. We now describe the method in detail, since it will serve as a basis for the proof of Theorem~\ref{veta_Z2_charakterizace}. 

Let $(G, \mathcal{X})$ be a given AT-graph and let $D$ be an arbitrary drawing of $G$. For example, we may place the vertices of $G$ on a circle in an arbitrary order and draw every edge as a straight-line segment. We use the ``obvious'' fact that every drawing of $G$ can be obtained from any other drawing of $G$ by a homeomorphism of the plane, which aligns the vertices of the two drawings, followed by a sequence of finitely many continuous deformations (isotopies) of the edges that keep the endpoints fixed, and maintain the property that every pair of edges have only finitely many points in common at each moment of the deformation (we will call such deformations \emph{generic}). This fact can be proved using the Jordan--Sch\"onflies theorem.

We will assume that the positions of the vertices are fixed. The algorithm will test whether the edges of the initial drawing $D$ can be continuously deformed to form an independent $\mathbb{Z}_2$-realization of $(G,\mathcal{X})$. During a generic continuous deformation from $D$ to some other drawing $D'$, three types of combinatorially interesting events can occur: 
\begin{enumerate}
\item[1)] two edges exchanging their order around their common vertex and creating a new crossing, 
\item[2)] an edge passing over an another edge, forming a pair of new crossings and a lens between them, 
\item[3)] an edge $e$ passing over a vertex $v$ not incident to $e$, creating a crossing with every edge incident to $v$. 
\end{enumerate}
Each of the three events also has a corresponding inverse event, where crossings are eliminated. Clearly, the effects on the parity of the number of crossings between edges are the same. The parity of the number of crossings between a pair of independent edges is affected only during the event 3) or its inverse, in which case we change the parity of the number of crossings of $e$ with all the edges incident to $v$; see Figure~\ref{obr_6_1_switch}. We call such an event an \emph{edge-vertex switch} and we will denote it by the ordered pair $(e,v)$. We will consider drawings up to the equivalence generated by events of type 1) and 2) and their inverses. 
%
Every edge-vertex switch $(e,v)$ can be performed independently of others, for any initial drawing, by deforming the edge $e$ along a curve connecting an interior point of $e$ with $v$.
We can thus represent the deformation from $D$ to $D'$ by the set of edge-vertex switches that were performed an odd number of times during the deformation.

\begin{figure}
\begin{center}
\epsfbox{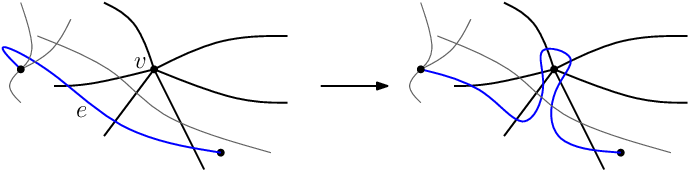}
\end{center}
\caption{A continuous deformation of $e$ resulting in an edge-vertex switch $(e,v)$.}
\label{obr_6_1_switch}
\end{figure}

A drawing $D$ of $G$ can then be represented by a vector ${\bf v}\in\mathbb{Z}_2^{M}$ where
$M$ is the number of unordered pairs of independent edges in $G$. The component of ${\bf v}$ corresponding to a pair $\{e,f\}$ is $1$ if $e$ and $f$ cross an odd number of times and $0$ otherwise. We remark that the space $\mathbb{Z}_2^{M}$ can also be considered as the space of subgraphs of the complement of the line graph $L(G)$. If $G=K_n$, then the complement of $L(G)$ is the Kneser graph $KG_{n,2}$.

Let $e$ be an edge of $G$ and $v$ a vertex of $G$ such that $v \notin e$. Performing an edge-vertex switch $(e,v)$ corresponds to adding the vector ${\bf w}_{(e,v)}\in\mathbb{Z}_2^{M}$ whose only components equal to $1$ are those indexed by pairs $\{e,f\}$ where $f$ is incident to $v$.
The set of all drawings of $G$ that can be obtained from $D$ by edge-vertex switches then corresponds to an affine subspace ${\bf v}+W$ where $W$ is the subspace generated by the set $\{{\bf w}_{(e,v)}; v\in V(G), e\in E(G), v \notin e\}$.

An AT-graph $(G,\mathcal{X})$ can similarly be represented by the vector ${\bf x}\in\mathbb{Z}_2^{M}$ whose component corresponding to a pair $\{e,f\}$ is $1$ if and only if $\{e,f\}\in \mathcal{X}$. In other words, ${\bf x}$ is the characteristic vector of $\mathcal{X}$. Now $(G,\mathcal{X})$ is independently $\mathbb{Z}_2$-realizable if and only if ${\bf x}\in {\bf v}+W$, equivalently, ${\bf x-v}\in W$. This is equivalent to the solvability of a system of $M$ linear equations over $\mathbb{Z}_2$, where each variable corresponds to one edge-vertex switch. In general, the vectors ${\bf w}_{(e,v)}$ are not linearly independent, so the number of variables in the system could be slightly reduced.

\subsection{Proof of Theorem~\ref{veta_Z2_charakterizace}}

We start with the following well-known fact, which implies that the condition of having no even $K_5$ and no odd $2K_3$ is necessary for independent $\mathbb{Z}_2$-realizability of complete AT-graphs.

\begin{observation}
\label{obs_K5_2K3}
The number of pairs of independent edges crossing an odd number of times is odd in every drawing of $K_{5}$ and even in every drawing of $2K_3$.
\end{observation}

\begin{proof}
For drawings of $K_5$ the observation follows from Lemma~\ref{lemma_parity_K5_K3_3}. For drawings of $2K_3$ we show a proof that is analogous to Kleitman's~\cite{Kle76_parity} proof for drawings of $K_5$.
We start with a plane drawing of $2K_3$, with no crossing. Every edge-vertex switch $(e,v)$ changes the parity of the number of crossings for exactly two pairs of edges, $\{e,f_1\}$ and $\{e,f_2\}$, where $f_1$ and $f_2$ are the two edges incident with $v$ but vertex-disjoint with $e$. Therefore, the number of independent odd-crossing pairs of edges remains even after an arbitrary sequence of edge-vertex switches. 
\end{proof}

It remains to show that complete AT-graphs with no even $K_5$ and no odd $2K_3$ are independently $\mathbb{Z}_2$-realizable.

Let $(K_n,\mathcal{X})$ be a complete AT-graph with no even $K_5$ and no odd $2K_3$. Let $M=3\cdot {n\choose 4}$ be the number of (unordered) pairs of independent edges in $K_n$. Let ${\bf x}\in \mathbb{Z}_2^M$ be the characteristic vector of $\mathcal{X}$, and let ${\bf v}\in \mathbb{Z}_2^M$ be the vector representing an arbitrary drawing of $K_n$. By Observation~\ref{obs_K5_2K3}, the vector ${\bf x}-{\bf v}$ represents a complete AT-graph with no odd $K_5$ and no odd $2K_3$; we will call these properties the \emph{$K_5$-property} and the \emph{$2K_3$-property}, respectively. 
%
Let $U \le\mathbb{Z}_2^M$ be the subspace of all vectors with the \emph{$K_5$-property} and the \emph{$2K_3$-property}. Let $W\le \mathbb{Z}_2^M$ be the subspace generated by the vectors ${\bf w}_{(e,v)}$, defined in the previous subsection. By the proof of Observation~\ref{obs_K5_2K3}, we have $W\le U$. We now show that $W=U$, which will imply the theorem.

Let ${\bf u}_0\in U$. By adding certain generators ${\bf w}_{(e,v)}$ to ${\bf u}_0$, we will gradually set all the coordinates to zero, in $2n-3$ steps, building a sequence of vectors ${\bf u}_1, {\bf u}_2, \dots, {\bf u}_{2n-3}$. In each step, we choose an edge $e$ of $K_n$ and set all the coordinates $\{e,f\}$ to zero. We will call this \emph{clearing} the edge $e$. Likewise, we call an edge $e$ \emph{clear} with respect to ${\bf u}_k$ if all the coordinates $\{e,f\}$ of ${\bf u}_k$ are zero.

Let $\{v_0, v_1, \dots, v_{n-1}\}$ be the vertex set of $K_n$. In steps $1,2,\dots,n-1$, we clear the edges $v_0v_1, v_0v_2, \dots, v_0v_{n-1}$, respectively. In step $i$, whenever a coordinate $\{v_0v_i,f\}$ of ${\bf u}_{i-1}$ is equal to $1$, we add the generator ${\bf w}_{(f,v_i)}$, which changes the coordinate $\{v_0v_i,f\}$ to $0$ but does not affect any other coordinate of the form $\{v_0v_j,g\}$. Thus, adding the necessary generators in step $i$ keeps the edges $v_0v_j$ clear for $j<i$. Hence, after the first $n-1$ steps, every edge incident to $v_0$ is clear with respect to ${\bf u}_{n-1}$.

In the remaining $n-2$ steps, we want to clear all the edges $v_1v_i$ for $i\ge 2$. Since we also want to keep all the edges incident to $v_0$ clear, we cannot just add the generator ${\bf w}_{(f,v_i)}$ if we want to set a coordinate $\{v_1v_i,f\}$ to zero. Thus, we will restrict ourselves to adding only those vectors of $W$ that have all the coordinates $\{v_0v_i,f\}$ equal to zero. For every $i,j\in\{1,2,\dots,n-1\}$, $i\neq j$, let
\[
{\bf y}_{i,j}={\bf w}_{(v_0v_i, v_j)} + \sum_{1\le k\le n-1;\ k\neq i,j} {\bf w}_{(v_jv_k, v_i)}.
\]
The nonzero coordinates of ${\bf y}_{i,j}$ are exactly those pairs $\{v_iv_k,v_jv_l\}$ where $k,l\in \{1,2,\dots,n-1\}$ and $i,j,k,l$ are distinct. In the remaining steps, we will be adding only the vectors ${\bf y}_{i,j}$. These vectors are a rather blunt tool compared to the vectors ${\bf w}_{(e,v)}$. Fortunately, they will be sufficient for clearing the edges $v_1v_i$ thanks to the $2K_3$-property and the $K_5$-property.

\begin{observation}
\label{obs_complete_bip}
If all the edges incident to $v_0$ are clear with respect to a vector~${\bf u}\in U$, then the set $E_{v_1v_i}$ of edges $f$ such that the coordinate $\{v_1v_i,f\}$ is nonzero in~${\bf u}$ forms a complete bipartite graph on the vertex set $\{v_2, v_3, \dots, v_{n-1}\} \setminus \{v_i\}$; this also includes the possibility of the empty graph.
\end{observation}

\begin{proof}
The $2K_3$-property applied to the triangle $v_0v_1v_i$ and every triangle induced by $\{v_2,\allowbreak v_3, \dots,\allowbreak v_{n-1}\} \setminus \{v_i\}$ implies that $E_{v_1v_i}$ has an even number of edges in every triangle induced by $\{v_2,\allowbreak v_3, \dots,\allowbreak v_{n-1}\} \setminus \{v_i\}$. This implies that each connected component of the complement $E_{v_1v_i}^c$ of $E_{v_1v_i}$ is a complete graph, and that the number of components of $E_{v_1v_i}^c$ is at most $2$. This is equivalent to $E_{v_1v_i}$ being complete bipartite or empty.
\end{proof}

Let $i\in\{2,3,\dots,n-1\}$ and suppose that the vector ${\bf u}_{n+i-3}$ has been computed. Let $E_{v_1v_i}$ be the set of edges $f$ such that the coordinate $\{v_1v_i,f\}$ is nonzero in~${\bf u}_{n+i-3}$. If $E_{v_1v_i}=\emptyset$, we set ${\bf u}_{n+i-2}={\bf u}_{n+i-3}$. Otherwise $E_{v_1v_i}$ is complete bipartite by Observation~\ref{obs_complete_bip}. Let $(A,B)$ be the bipartition of $\{v_2, v_3, \dots, v_{n-1}\} \setminus \{v_i\}$ that induces $E_{v_1v_i}$. 

\begin{observation}
All the vertices $\{v_2, v_3, \dots, v_{i-1}\}$ lie in the same part of the bipartition $(A,B)$.
\end{observation}

\begin{proof}
Suppose for contradiction that $v_j\in A$ and $v_k\in B$, for some $j,k\in \{2,3,\dots,i-1\}$. In the complete subgraph induced by the vertices $v_0,v_1,v_i,v_j,v_k$, only the edges $v_1v_i$ and $v_jv_k$ are not clear with respect to ${\bf u}_{n+i-3}$, and the pair $\{v_1v_i,v_jv_k\}$ is the only nonzero coordinate in ${\bf u}_{n+i-3}$ induced by the vertices $v_0,v_1,v_i,v_j,v_k$. This contradicts the $K_5$-property.
\end{proof}

Assume without loss of generality that $\{v_2, v_3, \dots, v_{i-1}\}\subseteq A$. We set
\[
{\bf u}_{n+i-2}={\bf u}_{n+i-3}+\sum_{v_j\in B} {\bf y}_{i,j}.
\]
All edges incident with $v_0$ are clear in ${\bf u}_{n+i-2}$, since this is true for all the summands. By adding $\sum_{v_j\in B} {\bf y}_{i,j}$, we have cleared the edge $v_1v_i$, while keeping the coordinates $\{v_1v_j,f\}$ unchanged for $j\in A$. Thus, all the edges $v_1v_j$ with $2\le j\le i$ are clear with respect to ${\bf u}_{n+i-2}$. Consequently, all edges incident with $v_0$ or $v_1$ are clear with respect to~${\bf u}_{2n-3}$. The following observation finishes the proof.

\begin{observation} 
We have ${\bf u}_{2n-3}={\bf 0}$. 
\end{observation}

\begin{proof}
Let $i,j,k,l\in \{2,3,\dots,n-1\}$ be distinct integers. Since all edges incident with $v_0$ or $v_1$ are clear with respect to ${\bf u}_{2n-3}$, the $2K_3$-property applied to the two triangles $v_0v_iv_k$ and $v_1v_jv_l$ implies that the coordinate $\{v_iv_k,v_jv_l\}$ is zero in~${\bf u}_{2n-3}$.
\end{proof}

\subsection{Concluding remarks}

\paragraph*{Minimality of the characterization.}
Here we show that none of the two conditions in Theorem~\ref{veta_Z2_charakterizace} can be omitted. This is equivalent to the fact that for $n$ large enough, there is a vector in $\mathbb{Z}_2^M$ with the $K_5$-property that does not have the $2K_3$-property, and another vector with the $2K_3$-property that does not have the $K_5$-property.

For $n\ge 5$, the AT-graph $(K_n,\emptyset)$ has no odd $2K_3$ but every subgraph isomorphic to $K_5$ is even.

For $n\ge 6$, the AT-graph $(K_n,\mathcal{X})$ where $\mathcal{X}$ is the set of all pairs of independent edges, has no even $K_5$ but every subgraph isomorphic to $2K_3$ is odd.

\paragraph*{Drawings of $K_6$.}
We conclude with the following interpretation of the characterization of independently $\mathbb{Z}_2$-realizable complete AT-graphs with six vertices. In the proof of Theorem~\ref{veta_Z2_charakterizace} we have shown that by edge-vertex switches, we can transform any drawing of $K_6$ to a drawing where all the edges incident to a chosen vertex $v_0$ cross every independent edge an even number of times. The drawings with this property can be represented as subgraphs of the Kneser graph $KG_{5,2}$, which is isomorphic to the Petersen graph. The $2K_3$-property now corresponds to the cycle space of the Petersen graph, and together with the $K_5$-property, they correspond to the even cycle space of the Petersen graph; that is, Eulerian subgraphs with even number of edges. The even cycle space of the Petersen graph has dimension $5$, and contains $10$ subgraphs isomorphic to $C_6$, $15$ subgraphs isomorphic to $C_8$, and $6$ subgraphs isomorphic to $2C_5$. However, the vectors characterizing the drawings of $K_6$ satisfy the $2K_3$ property but satisfy the ``opposite'' of the $K_5$-property; they correspond to the $32$ odd cycles in the Petersen graph, $12$ of them isomorphic to $C_5$ and the remaining $20$ isomorphic to $C_9$.

For $n=6$, the whole space $U$ of the vectors with the $K_5$-property and the $2K_3$-property has dimension $35$, since each of the $30$ edge-vertex switches $(e,v)$ with $v\neq v_0$ and $v_0\notin e$ changes the parity of the number of crossings for exactly one pair of edges of the type $\{v_0v_i,f\}$.

\section*{Acknowledgements}
I thank Martin Balko for his comments on an early version of the manuscript. 
I also thank the reviewers for their corrections and suggestions for improving the presentation. 


\begin{thebibliography}{99}  

\bibitem{Aetal15_allgood}
B. M. \'Abrego, O. Aichholzer, S. Fern\'andez-Merchant, T. Hackl, J. Pammer, A. Pilz, P. Ramos, G. Salazar and B. Vogtenhuber, All good drawings of small complete graphs, {\em EuroCG 2015, Book of abstracts}, 57--60, 2015.

\bibitem{A14_pers}
O. Aichholzer, personal communication, 2014.

\bibitem{AGR11_self}
L. Armas-Sanabria, F. Gonz{\'a}lez-Acu{\~n}a and J. Rodr{\'{\i}}guez-Viorato,
Self-intersection numbers of paths in compact surfaces,
{\em Journal of Knot Theory and Its Ramifications\/} {\bf 20}(3) (2011), 403--410.

\bibitem{BFK15_monotone}
M. Balko, R. Fulek and J. Kyn\v cl,
Crossing numbers and combinatorial characterization of monotone drawings of $K_n$, 
{\em Discrete and Computational Geometry\/} {\bf 53}(1) (2015), 107--143.

\bibitem{Ch11_facets}
M. Chimani,
Facets in the crossing number polytope,
{\em SIAM Journal on Discrete Mathematics\/} {\bf 25}(1) (2011), 95--111. 

\bibitem{FT_homeomorphisms}
B. Farb and B. Thurston, Homeomorphisms and simple closed curves, unpublished manuscript.

\bibitem{G05_complete}
E. Gioan,
Complete graph drawings up to triangle mutations,
{\em Graph-theoretic concepts in computer science},
{\em Lecture Notes in Computer Science\/} {\bf 3787}, 139--150, Springer, Berlin, 2005.

\bibitem{Ha34_uber}
H. Hanani, {\"{U}}ber wesentlich unpl{\"{a}}ttbare {K}urven im drei-dimensionalen {R}aume, 
{\em Fundamenta Mathematicae\/} {\bf 23} (1934), 135--142.

\bibitem{HM92_drawings}
H. Harborth and I. Mengersen,
Drawings of the complete graph with maximum number of crossings,
{\em Proceedings of the Twenty-third Southeastern International Conference on Combinatorics, Graph Theory, and Computing\/} 
({\em Boca Raton, FL\/}, 1992),
{\em Congressus Numerantium\/} {\bf 88} (1992), 225--228.

\bibitem{HS85_bigons}
J. Hass and P. Scott,
Intersections of curves on surfaces,
{\em Israel Journal of Mathematics\/} {\bf 51}(1-2) (1985), 90--120.

\bibitem{IMH82_matrix}
O.~H. Ibarra, S.~Moran and R.~Hui, A generalization of the fast {LUP} matrix
  decomposition algorithm and applications, {\em Journal of Algorithms\/} {\bf 3}(1)
  (1982), 45--56.

\bibitem{Jea06_LSP_decomp}
C.-P. Jeannerod,
LSP Matrix Decomposition Revisited, 
Laboratoire de l'Informatique du Parall\'elisme, \'Ecole Normale Sup\'erieure de Lyon, Research Report RR2006-28 (2006), \url{http://www.ens-lyon.fr/LIP/Pub/Rapports/RR/RR2006/RR2006-28.pdf}.

\bibitem{Kle76_parity}
D. J. Kleitman,
A note on the parity of the number of crossings of a graph,
{\em Journal of Combinatorial Theory, Series B\/} {\bf 21}(1) (1976), 88--89. 

\bibitem{K91_stringII} 
J. Kratochv\'\i l,
String graphs. II. Recognizing string graphs is NP-hard,
{\em Journal of Combinatorial Theory, Series B\/} {\bf 52} (1991), 67--78.

\bibitem{K91_noncrossing} 
J. Kratochv\'\i l, A. Lubiw and J. Ne\v{s}et\v{r}il,
Noncrossing subgraphs in topological layouts,
{\em SIAM Journal on Discrete Mathematics} {\bf 4}(2) (1991), 223--244.

\bibitem{KM89_np}
J. Kratochv\'\i l and J. Matou\v sek,
NP-hardness results for intersection graphs, 
{\em Commentationes Mathematicae Universitatis Carolinae} {\bf 30} (1989), 761--773.

\bibitem{K09_enumeration}
J. Kyn\v cl,
Enumeration of simple complete topological graphs,
{\em European Journal of Combinatorics\/} {\bf 30}(7) (2009), 1676--1685.

\bibitem{K11_simple_real}
J. Kyn\v cl,
Simple realizability of complete abstract topological graphs in P,
{\em Discrete and Computational Geometry\/} {\bf 45}(3) (2011), 383--399.

\bibitem{K13_improved}
J. Kyn\v cl, Improved enumeration of simple topological graphs,
{\em Discrete and Computational Geometry\/} {\bf 50}(3) (2013), 727--770.

\bibitem{LeGall14_powers}
F. Le Gall, Powers of tensors and fast matrix multiplication,
{\em I{SSAC} 2014---{P}roceedings of the 39th {I}nternational {S}ymposium on {S}ymbolic and {A}lgebraic {C}omputation}, 296--303, ACM, New York (2014).

\bibitem{M08_recent}
P. Mutzel,
Recent advances in exact crossing minimization (extended abstract),
{\em Electronic Notes in Discrete Mathematics\/} {\bf 31} (2008), 33--36.

\bibitem{PSS11_rotation}
M. J. Pelsmajer, M. Schaefer and D. {\v{S}}tefankovi{\v{c}},
Crossing numbers of graphs with rotation systems,
{\em Algorithmica\/} {\bf 60}(3) (2011), 679--702.

\bibitem{PT00_which}
J. Pach and G. T\'oth,
Which crossing number is it anyway?, 
{\em Journal of Combinatorial Theory, Series B\/} {\bf 80}(2) (2000), 225--246.

\bibitem{PT04_how}
J. Pach and G. T\'oth,
How many ways can one draw a graph?,
{\em Combinatorica\/} {\bf 26}(5) (2006), 559--576.
 
\bibitem{Sch13_hananitutte}
M. Schaefer, Hanani-{T}utte and related results, 
{\em Geometry---intuitive, discrete, and convex}, vol.~24 of {\em Bolyai Soc. Math. Stud.}, 259--299, J\'anos Bolyai Math. Soc., Budapest (2013).

\bibitem{Sch14_survey}
M. Schaefer, 
The graph crossing number and its variants: A survey, 
{\em Electronic Journal of Combinatorics}, Dynamic Survey 21 (2017).

\bibitem{SSS08_geometric}
M. Schaefer, E. Sedgwick and D. \v{S}tefankovi\v{c},
Computing Dehn twists and geometric intersection numbers in polynomial time,
{\em Proceedings of the 20th Canadian Conference on Computational Geometry (CCCG 2008)},
111--114, 2008.

\bibitem{Sze04_iocr}
L. A. Sz{\'e}kely, 
A successful concept for measuring non-planarity of graphs: the crossing number,
{\em Discrete Mathematics\/} {\bf 276}(1-3) (2004), 331--352.

\bibitem{Tutte70_toward}
W. T. Tutte, Toward a theory of crossing numbers, 
{\em Journal of Combinatorial Theory\/} {\bf 8} (1970), 45--53.

\bibitem{Williams12_faster}
V. V. Williams, Multiplying matrices faster than {C}oppersmith-{W}inograd, 
{\em Proceedings of the Forty-fourth Annual ACM Symposium on Theory of Computing}, STOC '12, 887--898, ACM, New York, NY, USA (2012).

\end{thebibliography}
\end{document}